\documentclass[a4paper]{amsart}

\usepackage{amsmath}
\usepackage{amssymb}
\usepackage{amsthm}
\usepackage[dvipdfmx]{graphicx}

\newtheorem{theorem}{Theorem}[section]
\newtheorem{proposition}[theorem]{Proposition}

\newtheorem{lemma}[theorem]{Lemma}
\theoremstyle{definition}
\newtheorem{definition}[theorem]{Definition}

\newtheorem{example}[theorem]{Example}

\newcommand{\oline}[1]{\mathbin{\overline{#1}}}
\newcommand{\uline}[1]{\mathbin{\underline{#1}}}

\begin{document}

\title{A multiple conjugation biquandle and handlebody-links}

\author[A. Ishii]{Atsushi Ishii}
\address{Institute of Mathematics, University of Tsukuba, 1-1-1 Tennodai, Tsukuba, Ibaraki 305-8571, Japan}
\email{aishii@math.tsukuba.ac.jp}

\author[M. Iwakiri]{Masahide Iwakiri}
\address{Graduate School of Science and Engineering, Saga University, 1 Honjo-machi, Saga-city, Saga, 840-8502, Japan}
\email{iwakiri@ms.saga-u.ac.jp}

\author[S. Kamada]{Seiichi Kamada}
\address{Department of Mathematics, Osaka City University, Sugimoto, Sumiyoshi-ku, Osaka 558-8585, Japan}
\email{skamada@sci.osaka-cu.ac.jp}

\author[J. Kim]{Jieon Kim}
\address{Osaka City University Advanced Mathematical Institute, Osaka City University, Sugimoto, Sumiyoshi-ku, Osaka 558-8585, Japan
}
\email{jieonkim@sci.osaka-cu.ac.jp}

\author[S. Matsuzaki]{Shosaku Matsuzaki}
\address{Department of Mathematics, School of Education, Waseda University, Nishi-Waseda 1-6-1, Shinjuku-ku, Tokyo, 169-8050, Japan}
\email{shosaku@aoni.waseda.jp}

\author[K. Oshiro]{Kanako Oshiro}
\address{Department of Information and Communication Sciences, Sophia University, 7-1 Kioi-cho, Chiyoda-ku, Tokyo 102-8554, Japan}
\email{oshirok@sophia.ac.jp}

\thanks{Atsushi Ishii was supported by JSPS KAKENHI Grant Number 15K0486. Seiichi Kamada was supported by JSPS KAKENHI Grant Number  26287013. Jieon Kim was supported by JSPS KAKENHI Grant Number 15F15319 and a JSPS Postdoctral Fellowship for Foreign Researchers. Kanako Oshiro  was supported by JSPS KAKENHI Grant Number 16K17600.
}

\date{}

\maketitle

\begin{abstract}
We introduce a multiple conjugation biquandle, and show that it is the universal algebra to define a semi-arc coloring invariant for handlebody-links.
A multiple conjugation biquandle is a generalization of a multiple conjugation quandle.
We extend the notion of $n$-parallel biquandle operations for any integer $n$, and show that any biquandle gives a multiple conjugation biquandle with them.
\end{abstract}

\section{Introduction}

A quandle~\cite{Joyce82,Matveev82}, biquandle~\cite{FennaJordanSantanaKauffman04, FennRourkeSanderson95, KauffmanRadford03}, and multiple conjugation quandle~\cite{Ishii15MCQ} are algebras having certain universal properties related to topological objects in geometric topology.
A quandle is a universal algebra to define an arc coloring invariant for oriented knots, where an arc coloring is a map from the set of arcs of a knot diagram to the algebra.
The axioms of a quandle correspond to the Reidemeister moves on oriented knot diagrams.
A biquandle is a generalization of a quandle, which is universal with respect to semi-arc colorings, and the axioms of a biquandle correspond to the Reidemeister moves.

A handlebody-knot is a handlebody embedded in the $3$-sphere $S^3$, whose diagram is given by a diagram of a spatial trivalent graph which is a spine of the handlebody.
A multiple conjugation quandle (MCQ) is a universal symmetric quandle with a partial multiplication to define arc coloring invariants for handlebody-knots, where a partial multiplication is an operation used at trivalent vertices (refer to \cite{Ishii15MCQ} or Section~\ref{sect:coloring}).
Some axioms of a multiple conjugation quandle are not directly derived from the Reidemeister moves.
In general we call conditions on an algebra which are directly derived from the Reidemeister moves primitive conditions.
In Section~4 of~\cite{Ishii15MCQ}, the first author listed primitive conditions for an arc coloring invariant and proved that the axioms of a multiple conjugation quandle are obtained from the primitive conditions.

In this paper, we introduce a multiple conjugation biquandle (MCB) as a universal biquandle with a partial multiplication to define semi-arc coloring invariants for handlebody-knots.
We list primitive conditions for a semi-arc coloring invariant and prove that the axioms of a multiple conjugation biquandle are obtained from the primitive conditions (Theorem~\ref{thm:universality}).
From the axioms of an MCB, it is naturally seen that an MCB is a generalization of an MCQ.
In~\cite{IshiiNelson16}, Nelson and the first author introduced a partially multiplicative biquandle to construct a semi-arc coloring invariant, whose axioms are almost identical to the primitive conditions.
Theorem~\ref{thm:universality} brings out the algebraic structure of a partially multiplicative biquandle.

In~\cite{IshiiNelson16}, the notions of $G$-family of biquandles and $n$-parallel biquandle operations were introduced for $n\in\mathbb{Z}_{\geq0}$.
We refine the axioms of a $G$-family of biquandles as a corollary of Theorem~\ref{thm:universality}, and extend the notion of $n$-parallel biquandle operations for any integer $n$.
We also show that, for any biquandle, the $n$-parallel biquandle operations yield a $\mathbb{Z}$-family of biquandles, which gives us a multiple conjugation biquandle.
We introduce a $G$-family of (generalized) Alexander biquandles, which also gives us many multiple conjugation biquandles.

(Co)homology theory is developed on quandles~\cite{CarterJelsovskyKamadaLangfordSaito03}, multiple conjugation quandles~\cite{CarterIshiiSaitoTanaka16}, and biquandles~\cite{CarterElhamdadiSaito04,CenicerosElhamdadiGreenNelson14}.
The theory provides quandle cocycle invariants, which give us various information about knots, surface-knots, and handlebody-knots (cf.~\cite{AsamiSatoh05, CarterElhamdadiSaitoSatoh06, CarterJelsovskyKamadaLangfordSaito03, IshiiIwakiriJangOshiro13, Iwakiri06, SatohShima04}).
(Co)homology theory will be also developed for multiple conjugation biquandles in the consecutive paper~\cite{IshiiIwakiriKamadaKimMatsuzakiOshiroHomology}.
This paper is the basis to develop the (co)homology theory for multiple conjugation biquandles.

The rest of the paper is organized as follows.
In Section~2, we recall the definition of a biquandle, and introduce $n$-parallel biquandle operations, whose well-definedness is given in Section~9.
In Section~3, we introduce a multiple conjugation biquandle with two equivalent definitions, and in Section~4, we show that the two definitions are equivalent.
In Section~5, we recall the definition of a handlebody-link, and introduce colorings for handlebody-knots.
In Section~6, we prove that a multiple conjugation biquandle gives a coloring invariant for handlebody-links.
In Sections~7 and~8, we discuss the universality of the algebras used for colorings.
In Section~9, we show some properties of $n$-parallel biquandle operations.

\section{Biquandles}

We recall the definition of a biquandle and introduce a conjugation biquandle.

\begin{definition}[\cite{FennRourkeSanderson95,KauffmanRadford03}]
A \textit{biquandle} is a non-empty set $X$ with binary operations $\uline{*},\oline{*}:X\times X\to X$ satisfying the following axioms.
\begin{itemize}
\item[(B1)]
For any $x\in X$, $x\uline{*}x=x\oline{*}x$.
\item[(B2)]
For any $a\in X$, the map $\uline{*}a:X\to X$ sending $x$ to $x\uline{*}a$ is bijective.
\item[]
For any $a\in X$, the map $\oline{*}a:X\to X$ sending $x$ to $x\oline{*}a$ is bijective.
\item[]
The map $S:X\times X\to X\times X$ defined by $S(x,y)=(y\oline{*}x,x\uline{*}y)$ is bijective.
\item[(B3)]
For any $x,y,z\in X$,
\begin{align*}
&(x\uline{*}y)\uline{*}(z\uline{*}y)=(x\uline{*}z)\uline{*}(y\oline{*}z), \\
&(x\uline{*}y)\oline{*}(z\uline{*}y)=(x\oline{*}z)\uline{*}(y\oline{*}z), \\
&(x\oline{*}y)\oline{*}(z\oline{*}y)=(x\oline{*}z)\oline{*}(y\uline{*}z).
\end{align*}
\end{itemize}
\end{definition}

We remark that $(X,*)$ is a quandle if and only if $(X,*,\oline{*})$ is a biquandle with $x\oline{*}y=x$.
We introduce a conjugation biquandle as an example of a biquandle.

\begin{definition}
Let $G$ be a group with identity element $e$, $\oline{*}:G\times G\to G$ a binary operation satisfying the following.
\begin{itemize}
\item
For any $a\in G$, $\oline{*}a:G\to G$ is a group homomorphism.
\item
For any $a,b,x\in G$, $x\oline{*}(ab)=(x\oline{*}a)\oline{*}(b\oline{*}a)$ and $x\oline{*}e=x$.
\end{itemize}
Define $a\uline{*}b:=(b^{-1}ab)\oline{*}b$.
Then $(G,\uline{*},\oline{*})$ is a biquandle.
We call it a \textit{$\oline{*}$-conjugation biquandle}, or just call it a \textit{conjugation biquandle}.
\end{definition}

It is easy to see that a $\oline{*}$-conjugation biquandle satisfies the conditions in Definition~\ref{def:MCB'}.
By Proposition~\ref{prop:two definitions}, we see that a $\oline{*}$-conjugation biquandle is a biquandle.

In this paper, we often omit brackets.
When we omit brackets, we apply binary operations from left on expressions, except for multiplications, which we always apply first.
For example, $a*_1b*_2cd*_3(e*_4f*_5g)$ stands for $((a*_1b)*_2(cd))*_3((e*_4f)*_5g)$, where $*_i$ is a binary operation.

We define $\uline{*}^na:=(\uline{*}a)^n$ and $\oline{*}^na:=(\oline{*}a)^n$ for $n\in\mathbb{Z}$.
Then $\uline{*}^{-1}a$ and $\oline{*}^{-1}a$ are the inverses of $\uline{*}a$ and $\oline{*}a$, respectively.
We also introduce $n$-parallel biquandle operations $\uline{*}^{[n]},\oline{*}^{[n]}$ for any integer $n$, which are extensions of the operations introduced in~\cite{IshiiNelson16}, where they were defined for $n\in\mathbb{Z}_{\geq0}$.

\begin{definition}
Let $X$ be a biquandle.
We define two families of binary operations $\uline{*}^{[n]},\oline{*}^{[n]}:X\times X\to X$ ($n\in\mathbb{Z}$) by the equalities
\begin{align}
&a\uline{*}^{[0]}b=a,
&&a\uline{*}^{[1]}b=a\uline{*}b,
&&a\uline{*}^{[i+j]}b=(a\uline{*}^{[i]}b)\uline{*}^{[j]}(b\uline{*}^{[i]}b), \label{eq:*u[n]def} \\
&a\oline{*}^{[0]}b=a,
&&a\oline{*}^{[1]}b=a\oline{*}b,
&&a\oline{*}^{[i+j]}b=(a\oline{*}^{[i]}b)\oline{*}^{[j]}(b\oline{*}^{[i]}b) \label{eq:*o[n]def}
\end{align}
for $i,j\in\mathbb{Z}$.
\end{definition}

In Section~\ref{sect:parallel}, we see that the binary operations $\uline{*}^{[n]}$ and $\oline{*}^{[n]}$ are well-defined.
Since
$a=a\uline{*}^{[0]}b
=(a\uline{*}^{[-1]}b)\uline{*}^{[1]}(b\uline{*}^{[-1]}b)
=(a\uline{*}^{[-1]}b)\uline{*}(b\uline{*}^{[-1]}b)$,
we have $a\uline{*}^{[-1]}b=a\uline{*}^{-1}(b\uline{*}^{[-1]}b)$.
Then we have the following by using \eqref{eq:*u[n]def}.
\begin{align*}
&a\uline{*}^{[0]}b=a, \hspace{5mm}
a\uline{*}^{[1]}b=a\uline{*}b, \hspace{5mm}
a\uline{*}^{[2]}b=(a\uline{*}b)\uline{*}(b\uline{*}b), \\
&a\uline{*}^{[3]}b=((a\uline{*}b)\uline{*}(b\uline{*}b))\uline{*}((b\uline{*}b)\uline{*}(b\uline{*}b)), \\
&a\uline{*}^{[-1]}b=a\uline{*}^{-1}(b\uline{*}^{[-1]}b), \hspace{5mm}
a\uline{*}^{[-2]}b=(a\uline{*}^{[-1]}b)\uline{*}^{[-1]}(b\uline{*}^{[-1]}b), \\
&a\uline{*}^{[-3]}b=((a\uline{*}^{[-1]}b)\uline{*}^{[-1]}(b\uline{*}^{[-1]}b))\uline{*}^{[-1]}((b\uline{*}^{[-1]}b)\uline{*}^{[-1]}(b\uline{*}^{[-1]}b)),
\end{align*}
where we note that $b\uline{*}^{[-1]}b$ is the unique element satisfying $(b\uline{*}^{[-1]}b)\uline{*}(b\uline{*}^{[-1]}b)=b$
(see Lemma~\ref{lem:x*x=y}).
We define the \textit{type} of a biquandle $X$ by
\[ \operatorname{type}X=\min\{n>0\,|\,a\uline{*}^{[n]}b=a=a\oline{*}^{[n]}b~(\forall a,b\in X)\}. \]
Any finite biquandle is of finite type~\cite{IshiiNelson16}.
For $m,n\in\mathbb{Z}$, if $\operatorname{type}X\mid(m-n)$, then $a\uline{*}^{[m]}b=a\uline{*}^{[n]}b$ and $a\oline{*}^{[m]}b=a\oline{*}^{[n]}b$, since we have
\begin{align*}
&a\uline{*}^{[i+\operatorname{type}X]}b
=(a\uline{*}^{[i]}b)\uline{*}^{[\operatorname{type}X]}(b\uline{*}^{[i]}b)
=a\uline{*}^{[i]}b, \\
&a\oline{*}^{[i+\operatorname{type}X]}b
=(a\oline{*}^{[i]}b)\oline{*}^{[\operatorname{type}X]}(b\oline{*}^{[i]}b)
=a\oline{*}^{[i]}b.
\end{align*}

We give examples of biquandles and their $n$-parallel biquandle operations below.

\begin{example}
Let $G$ be a group, and $X:=G^2$.
Fix $m,n\in\mathbb{Z}$.
We define
\begin{align*}
&(a_1,a_2)\uline{*}(b_1,b_2)=(b_1^{-n}a_1b_1^n,b_1^{-n}a_2b_1^n), \\
&(a_1,a_2)\oline{*}(b_1,b_2)=(a_1,b_1^{-n}b_2^{-m}a_2b_2^mb_1^n).
\end{align*}
Then $X$ is a biquandle.
We have
\begin{align*}
&(a_1,a_2)\uline{*}^{[k]}(b_1,b_2)=(b_1^{-kn}a_1b_1^{kn},b_1^{-kn}a_2b_1^{kn}), \\
&(a_1,a_2)\oline{*}^{[k]}(b_1,b_2)=(a_1,b_1^{-kn}b_2^{-km}a_2b_2^{km}b_1^{kn}).
\end{align*}
\end{example}

\begin{example}
Let $X$ be an $R[s^{\pm1},t^{\pm1}]$-module, where $R$ is a commutative ring.
We define $a\uline{*}b=ta+(s-t)b$, $a\oline{*}b=sa$.
Then $X$ is a biquandle, which we call an \textit{Alexander biquandle}.
We have $a\uline{*}^{[n]}b=t^na+(s^n-t^n)b$ and $a\oline{*}^{[n]}b=s^na$.
\end{example}

\begin{example}[\cite{Wada92}]
A group with the binary operations given in each of the following cases is a biquandle.
\begin{itemize}
\item[(1)]
$a\uline{*}b=a^{-1}$, $a\oline{*}b=a^{-1}$.
\item[(2)]
$a\uline{*}b=b^{-1}ab^{-1}$, $a\oline{*}b=a^{-1}$.
\item[(3)]
$a\uline{*}b=b^{-2}a$, $a\oline{*}b=b^{-1}a^{-1}b$.
\end{itemize}
We have
\begin{align*}
&a\uline{*}^{[n]}b=\begin{cases}
a\uline{*}b & \text{if $n$ is odd,} \\
a & \text{if $n$ is even,}
\end{cases}
&&a\oline{*}^{[n]}b=\begin{cases}
a\oline{*}b & \text{if $n$ is odd,} \\
a & \text{if $n$ is even}
\end{cases}
\end{align*}
for each case.
\end{example}

\begin{example}[\cite{KauffmanManturov05}]
Let $R:=\{a+bi+cj+dk\in\mathbb{H}\,|\,a,b,c,d\in\mathbb{Z}\}$, where $\mathbb{H}$ is the ring of quaternions with $i^2=j^2=k^2=ijk=-1$.
Let $X$ be an $R$-module.
We define $a\uline{*}b=-ja+(j+k)b$, $a\oline{*}b=ja+(k-j)b$.
Then $X$ is a biquandle.
We have
\begin{align*}
&a\uline{*}^{[n]}b=\begin{cases}
a & \text{if $n=4m$,} \\
-ja+(j+k)b & \text{if $n=4m+1$,} \\
-a & \text{if $n=4m+2$,} \\
ja-(j+k)b & \text{if $n=4m+3$,}
\end{cases} \\
&a\oline{*}^{[n]}b=\begin{cases}
a & \text{if $n=4m$,} \\
ja+(k-j)b & \text{if $n=4m+1$,} \\
-a & \text{if $n=4m+2$,} \\
-ja-(k-j)b & \text{if $n=4m+3$.}
\end{cases}
\end{align*}
\end{example}

We end this section with a lemma.

\begin{lemma} \label{lem:x*x=y}
Let $X$ be a biquandle.
\begin{itemize}
\item[(1)]
For $x,y\in X$, if $x\uline{*}y=y\oline{*}x$, then $x=y$.
\item[(2)]
For any $a\in X$, there exists a unique element $\alpha\in X$ such that $\alpha\uline{*}\alpha=\alpha\oline{*}\alpha=a$.
\end{itemize}
\end{lemma}

\begin{proof}
\begin{itemize}
\item[(1)]
We have $x=y$ from
\begin{align*}
&x\oline{*}x
\overset{\rm(B1)}{=}x\uline{*}x
=(x\uline{*}x)\uline{*}(y\oline{*}x)\uline{*}^{-1}(y\oline{*}x) \\
&\overset{\rm(B3)}{=}
(x\uline{*}y)\uline{*}(x\uline{*}y)\uline{*}^{-1}(y\oline{*}x)
=(y\oline{*}x)\uline{*}(y\oline{*}x)\uline{*}^{-1}(y\oline{*}x)
=y\oline{*}x.
\end{align*}

\item[(2)]
By axiom (B2), there exists a unique pair $(\alpha_1,\alpha_2)\in X$ such that $(\alpha_2\oline{*}\alpha_1,\alpha_1\uline{*}\alpha_2)=(a,a)$.
Since $\alpha_1\uline{*}\alpha_2=\alpha_2\oline{*}\alpha_1$ implies $\alpha_1=\alpha_2$, we put $\alpha:=\alpha_1=\alpha_2$.
Then $\alpha$ is a unique element satisfying
$\alpha\uline{*}\alpha=\alpha\oline{*}\alpha=a$.
\end{itemize}
\end{proof}

\section{A multiple conjugation biquandle (MCB)}

In this section, we introduce the notion of a multiple conjugation biquandle (MCB).
We give two equivalent definitions for the multiple conjugation biquandle.
The first one is useful to study coloring invariants, and the second one is useful to check that a given algebra is a multiple conjugation biquandle.
In the next section, we see that these two definitions are equivalent.

Let $X$ be the disjoint union of groups $G_\lambda$ ($\lambda\in\Lambda$).
We denote by $G_a$ the group $G_\lambda$ to which $a\in X$ belongs.
We denote by $e_\lambda$ the identity of $G_\lambda$.
We also denote it by $e_a$ if $a\in G_\lambda$.
The identity of $G_a$ is the element $e_a$.

\begin{definition} \label{def:MCB}
A \textit{multiple conjugation biquandle} is a biquandle $(X,\uline{*},\oline{*})$ which is the disjoint union of groups $G_\lambda$ ($\lambda\in\Lambda$) satisfying the following axioms.
\begin{itemize}
\item
For any $a,x\in X$, $\uline{*}x:G_a\to G_{a\uline{*}x}$ and $\oline{*}x:G_a\to G_{a\oline{*}x}$ are group homomorphisms.
\item
For any $a,b\in G_\lambda$ and $x\in X$,
\begin{align}
&x\uline{*}ab=(x\uline{*}a)\uline{*}(b\oline{*}a), \label{eq:x*u(ab)} \\
&x\oline{*}ab=(x\oline{*}a)\oline{*}(b\oline{*}a), \label{eq:x*o(ab)} \\
&a^{-1}b\oline{*}a=ba^{-1}\uline{*}a. \label{eq:R14}
\end{align}
\end{itemize}
\end{definition}

\begin{definition} \label{def:MCB'}
A \textit{multiple conjugation biquandle} $X$ is the disjoint union of groups $G_\lambda$ ($\lambda\in\Lambda$) with binary operations $\uline{*},\oline{*}:X\times X\to X$ satisfying the following axioms.
\begin{itemize}
\item
For any $x,y,z\in X$,
\begin{align}
&(x\uline{*}y)\uline{*}(z\uline{*}y)=(x\uline{*}z)\uline{*}(y\oline{*}z), \label{eq:R3-1'} \\
&(x\uline{*}y)\oline{*}(z\uline{*}y)=(x\oline{*}z)\uline{*}(y\oline{*}z), \label{eq:R3-2'} \\
&(x\oline{*}y)\oline{*}(z\oline{*}y)=(x\oline{*}z)\oline{*}(y\uline{*}z). \label{eq:R3-3'}
\end{align}
\item
For any $a,x\in X$, $\uline{*}x:G_a\to G_{a\uline{*}x}$ and $\oline{*}x:G_a\to G_{a\oline{*}x}$ are group homomorphisms.
\item
For any $a,b\in G_\lambda$ and $x\in X$,
\begin{align}
&x\uline{*}ab=(x\uline{*}a)\uline{*}(b\oline{*}a),
\hspace{1em}x\uline{*}e_\lambda=x, \label{eq:x*u(ab)'} \\
&x\oline{*}ab=(x\oline{*}a)\oline{*}(b\oline{*}a),
\hspace{1em}x\oline{*}e_\lambda=x, \label{eq:x*o(ab)'} \\
&a^{-1}b\oline{*}a=ba^{-1}\uline{*}a. \label{eq:R14'}
\end{align}
\end{itemize}
\end{definition}

We remark that a multiple conjugation biquandle consisting of one group is a conjugation biquandle.
A $G$-family of biquandles yields a multiple conjugation biquandle.
We recall the definition of a $G$-family of biquandles below, where the bijectivity in its original axioms in~\cite{IshiiNelson16} is replaced with $x\uline{*}^ey=x\oline{*}^ey=x$.
This refinement is induced from the equivalence of the two definitions of a multiple conjugation biquandle.
For details on a $G$-family of biquandles, we refer the reader to~\cite{IshiiNelson16}.

\begin{definition}
Let $G$ be a group with identity element $e$.
A \textit{$G$-family of biquandles} is a non-empty set $X$ with two families of binary operations $\uline{*}^g,\oline{*}^g:X\times X\to X$ ($g\in G$) satisfying the following axioms.
\begin{itemize}
\item
For any $x,y,z\in X$ and $g,h\in G$,
\begin{align*}
&(x\uline{*}^gy)\uline{*}^h(z\oline{*}^gy)
=(x\uline{*}^hz)\uline{*}^{h^{-1}gh}(y\uline{*}^hz), \\
&(x\oline{*}^gy)\uline{*}^h(z\oline{*}^gy)
=(x\uline{*}^hz)\oline{*}^{h^{-1}gh}(y\uline{*}^hz), \\
&(x\oline{*}^gy)\oline{*}^h(z\oline{*}^gy)
=(x\oline{*}^hz)\oline{*}^{h^{-1}gh}(y\uline{*}^hz).
\end{align*}
\item
For any $x,y\in X$ and $g,h\in G$,
\begin{align*}
&x\uline{*}^{gh}y=(x\uline{*}^gy)\uline{*}^h(y\uline{*}^gy),
\hspace{1em}x\uline{*}^ey=x, \\
&x\oline{*}^{gh}y=(x\oline{*}^gy)\oline{*}^h(y\oline{*}^gy),
\hspace{1em}x\oline{*}^ey=x, \\
&x\uline{*}^gx=x\oline{*}^gx.
\end{align*}
\end{itemize}
\end{definition}

\begin{proposition}[\cite{IshiiNelson16}]
Let $(X,(\uline{*}^g)_{g\in G},(\oline{*}^g)_{g\in G})$ be a $G$-family of biquandles.
Then $X\times G=\bigsqcup_{x\in X}\{x\}\times G$ is a multiple conjugation biquandle with the binary operations $\uline{*},\oline{*}:(X\times G)\times(X\times G)\to X\times G$ defined by
\begin{align*}
&(x,g)\uline{*}(y,h)=(x\uline{*}^hy,h^{-1}gh), &&(x,g)\oline{*}(y,h)=(x\oline{*}^hy,g).
\end{align*}
\end{proposition}

We call this multiple conjugation biquandle the \textit{associated multiple conjugation biquandle}.

A biquandle turns into a $G$-family of biquandles with parallel biquandle operations $\uline{*}^{[n]},\oline{*}^{[n]}$ (see Proposition~\ref{prop:Z-family}).
Therefore we can construct a multiple conjugation biquandle from any biquandle.
We introduce a $G$-family of (generalized) Alexander biquandles in the following proposition.

\begin{proposition}
Let $G$ be a group with identity $e$, and let $\varphi:G\to Z(G)$ be a homomorphism, where $Z(G)$ is the center of $G$.
\begin{itemize}
\item[(1)]
Let $X$ be a group with a right action of $G$.
We denote by $x^g$ the result of $g$ acting on $x$.
We define binary operations $\uline{*}^g,\oline{*}^g:X\times X\to X$ by $x\uline{*}^gy=(xy^{-1})^gy^{\varphi(g)}$, $x\oline{*}^gy=x^{\varphi(g)}$.
Then $X$ is a $G$-family of biquandles, which we call a $G$-family of generalized Alexander biquandles.

\item[(2)]
Let $R$ be a ring, $X$ a right $R[G]$-module, where $R[G]$ is the group ring of $G$ over $R$.
We define binary operations $\uline{*}^g,\oline{*}^g:X\times X\to X$ by $x\uline{*}^gy=xg+y(\varphi(g)-g)$, $x\oline{*}^gy=x\varphi(g)$.
Then $X$ is a $G$-family of biquandles, which we call a $G$-family of Alexander biquandles.
\end{itemize}
\end{proposition}

\begin{proof}
It is sufficient to show (1), since (2) follows from (1) with an abelian group $X$.
For any $x,y,z\in X$ and $g,h\in G$, we have
\begin{align*}
(x\uline{*}^gy)\uline{*}^h(z\oline{*}^gy)
&=x^{gh}y^{-gh}y^{\varphi(g)h}z^{-\varphi(g)h}z^{\varphi(g)\varphi(h)} \\
&=(x\uline{*}^hz)\uline{*}^{h^{-1}gh}(y\uline{*}^hz), \\
(x\oline{*}^gy)\uline{*}^h(z\oline{*}^gy)
&=x^{\varphi(g)h}z^{-\varphi(g)h}z^{\varphi(g)\varphi(h)}
=(x\uline{*}^hz)\oline{*}^{h^{-1}gh}(y\uline{*}^hz), \\
(x\oline{*}^gy)\oline{*}^h(z\oline{*}^gy)
&=x^{\varphi(g)\varphi(h)}
=(x\oline{*}^hz)\oline{*}^{h^{-1}gh}(y\uline{*}^hz)
\end{align*}
and
\begin{align*}
&x\uline{*}^{gh}y
=(xy^{-1})^{gh}y^{\varphi(g)\varphi(h)}
=(x\uline{*}^gy)\uline{*}^h(y\uline{*}^gy),
&&x\uline{*}^ey=x, \\
&x\oline{*}^{gh}y
=x^{\varphi(g)\varphi(h)}
=(x\oline{*}^gy)\oline{*}^h(y\oline{*}^gy),
&&x\oline{*}^ey=x, \\
&x\uline{*}^gx=x^{\varphi(g)}=x\oline{*}^gx,
\end{align*}
where $x^{-g}$ denotes $(x^g)^{-1}$, which coincides with $(x^{-1})^g$.
%
\end{proof}

\section{The two definitions are equivalent}

In this section, we see that the two definitions of a multiple conjugation biquandle introduced in the previous section are equivalent.

\begin{lemma} \label{lem:x*e=x}
Let $X=\bigsqcup_{\lambda\in\Lambda}G_\lambda$ be a multiple conjugation biquandle in the sense of Definition~\ref{def:MCB}.
\begin{itemize}
\item[(1)]
For any $x\in X$ and $\lambda\in\Lambda$,
\begin{align}
&x\uline{*}e_\lambda=x, &&x\oline{*}e_\lambda=x. \label{eq:x*e}
\end{align}

\item[(2)]
For any $a,x\in X$, $\uline{*}x:G_a\to G_{a\uline{*}x}$ and $\oline{*}x:G_a\to G_{a\oline{*}x}$ are bijections.
Furthermore, $\uline{*}^{-1}x=\uline{*}(x^{-1}\oline{*}x)$, $\oline{*}^{-1}x=\oline{*}(x^{-1}\oline{*}x)$.

\item[(3)]
For any $b\in X$, the map $f:G_b\to G_{b\oline{*}b}$ which sends $x$ to $x^{-1}b\oline{*}x$ is bijective.
Furthermore, its inverse $f^{-1}:G_{b\oline{*}b}\to G_b$ is given by $f^{-1}(x)=b(x^{-1}\oline{*}x\oline{*}^{-1}b)$.
\end{itemize}
\end{lemma}

\begin{proof}
\begin{itemize}
\item[(1)]
Let $\alpha\in X$ be the unique element satisfying
$\alpha\uline{*}\alpha=\alpha\oline{*}\alpha=e_\lambda$.
Then
\begin{align*}
&x\uline{*}e_\lambda
=x\uline{*}e_{\alpha\oline{*}\alpha}
=((x\uline{*}^{-1}\alpha)\uline{*}\alpha)\uline{*}(e_\alpha\oline{*}\alpha)
\overset{\eqref{eq:x*u(ab)}}{=}
(x\uline{*}^{-1}\alpha)\uline{*}\alpha e_\alpha
=x, \\
&x\oline{*}e_\lambda
=x\oline{*}e_{\alpha\oline{*}\alpha}
=((x\oline{*}^{-1}\alpha)\oline{*}\alpha)\oline{*}(e_\alpha\oline{*}\alpha)
\overset{\eqref{eq:x*o(ab)}}{=}
(x\oline{*}^{-1}\alpha)\oline{*}\alpha e_\alpha
=x.
\end{align*}

\item[(2)]
Since the maps $\uline{*}x:X\to X$, $\oline{*}x:X\to X$ are bijective, it is sufficient to show that
\[ b\uline{*}x\in G_{a\uline{*}x}
\Leftrightarrow b\in G_a
\Leftrightarrow b\oline{*}x\in G_{a\oline{*}x}. \]
We have $b\in G_a\Rightarrow b\uline{*}x\in G_{a\uline{*}x}$ and $b\in G_a\Rightarrow b\oline{*}x\in G_{a\oline{*}x}$ by the well-definedness of the maps $\uline{*}x:G_a\to G_{a\uline{*}x}$ and $\oline{*}x:G_a\to G_{a\oline{*}x}$, respectively.
We have $b\uline{*}x\in G_{a\uline{*}x}\Rightarrow b\in G_a$ and $b\oline{*}x\in G_{a\oline{*}x}\Rightarrow b\in G_a$ by the equalities
\begin{align*}
&(a\uline{*}x)\uline{*}(x^{-1}\oline{*}x)=a
=(a\oline{*}x)\oline{*}(x^{-1}\oline{*}x), \\
&(b\uline{*}x)\uline{*}(x^{-1}\oline{*}x)=b
=(b\oline{*}x)\oline{*}(x^{-1}\oline{*}x),
\end{align*}
which follow from
\begin{align*}
&(y\uline{*}x)\uline{*}(x^{-1}\oline{*}x)
\overset{\eqref{eq:x*u(ab)}}{=}y\uline{*}xx^{-1}
=y\uline{*}e_x
\overset{\eqref{eq:x*e}}{=}y, \\
&(y\oline{*}x)\oline{*}(x^{-1}\oline{*}x)
\overset{\eqref{eq:x*o(ab)}}{=}y\oline{*}xx^{-1}
=y\oline{*}e_x
\overset{\eqref{eq:x*e}}{=}y
\end{align*}
for any $y\in X$.

\item[(3)]
Let $g:G_{b\oline{*}b}\to G_b$ be the map defined by $g(x)=b(x^{-1}\oline{*}x\oline{*}^{-1}b)$, which is well-defined, since
\[ (e_b\oline{*}b)^{-1}\oline{*}(e_b\oline{*}b)\oline{*}^{-1}b
=(e_b\oline{*}b)\oline{*}(e_b\oline{*}b)\oline{*}^{-1}b
\overset{\eqref{eq:x*o(ab)}}{=}e_b\oline{*}b\oline{*}^{-1}b
=e_b\in G_b. \]
Then we have
\begin{align*}
(g\circ f)(x)
&=b((x^{-1}b\oline{*}x)^{-1}\oline{*}(x^{-1}b\oline{*}x)\oline{*}^{-1}b) \\
&=b((b^{-1}x\oline{*}x)\oline{*}(x^{-1}b\oline{*}x)\oline{*}^{-1}b) \\
&\overset{\eqref{eq:x*o(ab)}}{=}b(b^{-1}x\oline{*}b\oline{*}^{-1}b)
=x, \\
(f\circ g)(x)
&=(b(x^{-1}\oline{*}x\oline{*}^{-1}b))^{-1}b\oline{*}(b(x^{-1}\oline{*}x\oline{*}^{-1}b)) \\
&=(x\oline{*}x\oline{*}^{-1}b)\oline{*}(b(x^{-1}\oline{*}x\oline{*}^{-1}b)) \\
&\overset{\eqref{eq:x*o(ab)}}{=}
((x\oline{*}x\oline{*}^{-1}b)\oline{*}b)\oline{*}((x^{-1}\oline{*}x\oline{*}^{-1}b)\oline{*}b) \\
&=(x\oline{*}x)\oline{*}(x^{-1}\oline{*}x)
\overset{\eqref{eq:x*o(ab)}}{=}x\oline{*}e_x
\overset{\eqref{eq:x*e}}{=}x.
\end{align*}
\end{itemize}
\end{proof}

\begin{proposition} \label{prop:two definitions}
Let $X$ be the disjoint union of groups $G_\lambda$ ($\lambda\in\Lambda$) with binary operations $\uline{*},\oline{*}:X\times X\to X$.
Then $X$ is an MCB in the sense of Definition~\ref{def:MCB} if and only if $X$ is an MCB in the sense of Definition~\ref{def:MCB'}.
\end{proposition}

\begin{proof}
By Lemma~\ref{lem:x*e=x}, it is sufficient to show the ``if'' part.
For any $x\in X$, we have
\begin{align}
x\uline{*}x=x^2x^{-1}\uline{*}x
\overset{\eqref{eq:R14'}}{=}x^{-1}x^2\oline{*}x=x\oline{*}x. \label{eq:R1'}
\end{align}

The map $\uline{*}(a^{-1}\oline{*}a):X\to X$ is the inverse of $\uline{*}a:X\to X$, since we have
\begin{align}
(x\uline{*}a)\uline{*}(a^{-1}\oline{*}a)
\overset{\eqref{eq:x*u(ab)'}}{=}x\uline{*}aa^{-1}=x\uline{*}e_a
\overset{\eqref{eq:x*u(ab)'}}{=}x, \label{eq:x*uaaa=x}
\end{align}
and
\begin{align*}
x\uline{*}(a^{-1}\oline{*}a)\uline{*}a
&\overset{\eqref{eq:x*uaaa=x}}{=}
(x\uline{*}(a^{-1}\oline{*}a))\uline{*}((a\uline{*}a)\uline{*}(a^{-1}\oline{*}a)) \\
&\overset{\eqref{eq:R3-1'}}{=}
(x\uline{*}(a\uline{*}a))\uline{*}((a^{-1}\oline{*}a)\oline{*}(a\uline{*}a)) \\
&\overset{\eqref{eq:R1'}}{=}
(x\uline{*}(a\oline{*}a))\uline{*}((a\oline{*}a)^{-1}\oline{*}(a\oline{*}a)) \\
&\overset{\eqref{eq:x*uaaa=x}}{=}x.
\end{align*}
Therefore the map $\uline{*}a:X\to X$ is bijective.

The map $\oline{*}(a^{-1}\oline{*}a):X\to X$ is the inverse of $\oline{*}a:X\to X$, since we have
\begin{align}
(x\oline{*}a)\oline{*}(a^{-1}\oline{*}a)
\overset{\eqref{eq:x*o(ab)'}}{=}x\oline{*}aa^{-1}=x\oline{*}e_a
\overset{\eqref{eq:x*o(ab)'}}{=}x, \label{eq:x*oaaa=x}
\end{align}
and
\begin{align*}
x\oline{*}(a^{-1}\oline{*}a)\oline{*}a
&\overset{\eqref{eq:x*uaaa=x}}{=}
(x\oline{*}(a^{-1}\oline{*}a))\oline{*}((a\uline{*}a)\uline{*}(a^{-1}\oline{*}a)) \\
&\overset{\eqref{eq:R3-3'}}{=}
(x\oline{*}(a\uline{*}a))\oline{*}((a^{-1}\oline{*}a)\oline{*}(a\uline{*}a)) \\
&\overset{\eqref{eq:R1'}}{=}
(x\oline{*}(a\oline{*}a))\oline{*}((a\oline{*}a)^{-1}\oline{*}(a\oline{*}a)) \\
&\overset{\eqref{eq:x*oaaa=x}}{=}x.
\end{align*}
Therefore the map $\oline{*}a:X\to X$ is bijective.

We show that the map $S:X\times X\to X\times X$ defined by $S(x,y)=(y\oline{*}x,x\uline{*}y)$ is the bijection whose inverse $T:X\times X\to X\times X$ is given by
\[ T(x,y)=(y\uline{*}(x\oline{*}x\oline{*}^{-1}y)^{-1},x\oline{*}(y\uline{*}y\uline{*}^{-1}x)^{-1}), \]
where we note that
\begin{align*}
y\uline{*}(x\oline{*}x\oline{*}^{-1}y)^{-1}
&=((y\oline{*}y)\oline{*}^{-1}y)\uline{*}((x^{-1}\oline{*}x)\oline{*}^{-1}y) \\
&=((y\oline{*}y)\oline{*}(y^{-1}\oline{*}y))\uline{*}((x^{-1}\oline{*}x)\oline{*}(y^{-1}\oline{*}y)) \\
&\overset{\eqref{eq:R3-2'}}{=}
((y\oline{*}y)\uline{*}(x^{-1}\oline{*}x))\oline{*}((y^{-1}\oline{*}y)\uline{*}(x^{-1}\oline{*}x)) \\
&=((y\oline{*}y)\uline{*}^{-1}x)\oline{*}((y^{-1}\oline{*}y)\uline{*}^{-1}x) \\
&=(y\oline{*}y\uline{*}^{-1}x)\oline{*}(y\oline{*}y\uline{*}^{-1}x)^{-1} \\
&\overset{\eqref{eq:R1'}}{=}
(y\uline{*}y\uline{*}^{-1}x)\oline{*}(y\uline{*}y\uline{*}^{-1}x)^{-1}
\end{align*}
and
\begin{align*}
x\oline{*}(y\uline{*}y\uline{*}^{-1}x)^{-1}
&\overset{\eqref{eq:R1'}}{=}x\oline{*}(y\oline{*}y\uline{*}^{-1}x)^{-1} \\
&=((x\uline{*}x)\uline{*}^{-1}x)\oline{*}((y^{-1}\oline{*}y)\uline{*}^{-1}x) \\
&=((x\uline{*}x)\uline{*}(x^{-1}\oline{*}x))\oline{*}((y^{-1}\oline{*}y)\uline{*}(x^{-1}\oline{*}x)) \\
&\overset{\eqref{eq:R3-2'}}{=}
((x\uline{*}x)\oline{*}(y^{-1}\oline{*}y))\uline{*}((x^{-1}\oline{*}x)\oline{*}(y^{-1}\oline{*}y)) \\
&=((x\uline{*}x)\oline{*}^{-1}y)\uline{*}((x^{-1}\oline{*}x)\oline{*}^{-1}y) \\
&=(x\uline{*}x\oline{*}^{-1}y)\uline{*}(x\oline{*}x\oline{*}^{-1}y)^{-1} \\
&\overset{\eqref{eq:R1'}}{=}
(x\oline{*}x\oline{*}^{-1}y)\uline{*}(x\oline{*}x\oline{*}^{-1}y)^{-1}.
\end{align*}
Then $T\circ S=\mathrm{id}_{X\times X}$ and $S\circ T=\mathrm{id}_{X\times X}$ follow from
\begin{align*}
&(x\uline{*}y)\uline{*}((y\oline{*}x)\oline{*}(y\oline{*}x)\oline{*}^{-1}(x\uline{*}y))^{-1} \\
&\overset{\eqref{eq:R3-3'}}{=}
(x\uline{*}y)\uline{*}((y\oline{*}y)\oline{*}(x\uline{*}y)\oline{*}^{-1}(x\uline{*}y))^{-1} \\
&=(x\uline{*}y)\uline{*}(y\oline{*}y)^{-1}
=(x\uline{*}y)\uline{*}(y^{-1}\oline{*}y)
\overset{\eqref{eq:x*uaaa=x}}{=}x, \\
&(y\oline{*}x)\oline{*}((x\uline{*}y)\uline{*}(x\uline{*}y)\uline{*}^{-1}(y\oline{*}x))^{-1} \\
&\overset{\eqref{eq:R3-1'}}{=}
(y\oline{*}x)\oline{*}((x\uline{*}x)\uline{*}(y\oline{*}x)\uline{*}^{-1}(y\oline{*}x))^{-1} \\
&\overset{\eqref{eq:R1'}}{=}(y\oline{*}x)\oline{*}(x\oline{*}x)^{-1}
=(y\oline{*}x)\oline{*}(x^{-1}\oline{*}x)
\overset{\eqref{eq:x*oaaa=x}}{=}y
\end{align*}
and
\begin{align*}
&(x\oline{*}(y\uline{*}y\uline{*}^{-1}x)^{-1})\oline{*}(y\uline{*}(x\oline{*}x\oline{*}^{-1}y)^{-1}) \\
&=(x\oline{*}(y\uline{*}y\uline{*}^{-1}x)^{-1})\oline{*}((y\uline{*}y\uline{*}^{-1}x)\oline{*}(y\uline{*}y\uline{*}^{-1}x)^{-1})
\overset{\eqref{eq:x*oaaa=x}}{=}x, \\
&(y\uline{*}(x\oline{*}x\oline{*}^{-1}y)^{-1})\uline{*}(x\oline{*}(y\uline{*}y\uline{*}^{-1}x)^{-1}) \\
&=(y\uline{*}(x\oline{*}x\oline{*}^{-1}y)^{-1})\uline{*}((x\oline{*}x\oline{*}^{-1}y)\uline{*}(x\oline{*}x\oline{*}^{-1}y)^{-1}) \\
&\overset{\eqref{eq:R3-1'}}{=}
(y\uline{*}(x\oline{*}x\oline{*}^{-1}y))\uline{*}((x\oline{*}x\oline{*}^{-1}y)^{-1}\oline{*}(x\oline{*}x\oline{*}^{-1}y))
\overset{\eqref{eq:x*uaaa=x}}{=}y,
\end{align*}
respectively.
This completes the proof.
\end{proof}

\section{MCB colorings for handlebody-links}
\label{sect:coloring}

In this section we recall a diagrammatic presentation of a handlebody-link and consider its colorings using a multiple conjugation biquandle.

A \textit{handlebody-link} is the disjoint union of handlebodies embedded in the $3$-sphere $S^3$.
In this paper, we assume that every component of a handlebody-link is of genus at least $1$.
An \textit{$S^1$-orientation} of a handlebody-link is a collection of $S^1$-orientations of all genus-$1$ components, that are solid tori, of the handlebody-link.
Here an $S^1$-orientation of a solid torus means an orientation of its core $S^1$.
Two $S^1$-oriented handlebody-links are \textit{equivalent} if there is an orientation-preserving self-homeomorphism of $S^3$ which sends one to the other preserving the $S^1$-orientation.

A \textit{Y-orientation} of a trivalent graph $G$, whose vertices are of valency $3$, is a direction of all edges of $G$ satisfying that every vertex of $G$ is both the initial vertex of a directed edge and the terminal vertex of a directed edge (See Figure \ref{fig:Y-orientations}).
In this paper, a trivalent graph may have a circle component, which has no vertices.

\begin{figure}
\center
\begin{minipage}{40pt}
\begin{picture}(50,40)
 \put(20,20){\vector(0,-1){20}}
 \put(0,40){\vector(1,-1){20}}
 \put(40,40){\vector(-1,-1){20}}
\end{picture}
\end{minipage}
\hspace{5em}
\begin{minipage}{40pt}
\begin{picture}(50,40)
 \put(20,40){\vector(0,-1){20}}
 \put(20,20){\vector(-1,-1){20}}
 \put(20,20){\vector(1,-1){20}}
\end{picture}
\end{minipage}
\caption{Y-orientations}
\label{fig:Y-orientations}
\end{figure}
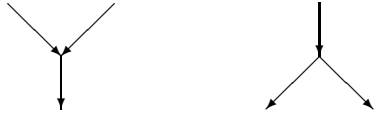

A finite graph embedded in $S^3$ is called a \textit{spatial graph}.
For a Y-oriented spatial trivalent graph $K$ and an $S^1$-oriented handlebody-link $H$, we say that $K$ \textit{represents} $H$ if $H$ is a regular neighborhood of $K$ and the $S^1$-orientation of $H$ agrees with the Y-orientation.
Then any $S^1$-oriented handlebody-link can be represented by some Y-oriented spatial trivalent graph.
The following theorem plays a fundamental role in constructing $S^1$-oriented handlebody-link invariants.

\begin{theorem}[\cite{Ishii15Markov}] \label{thm:ReidemeisterMoves}
For a diagram $D_i$ of a Y-oriented spatial trivalent graph $K_i$ $(i=1,2)$, $K_1$ and $K_2$ represent an equivalent $S^1$-oriented handlebody-link if and only if $D_1$ and $D_2$ are related by a finite sequence of R1--R6 moves depicted in Figure~\ref{fig:ReidemeisterMoves} preserving Y-orientations.
\end{theorem}

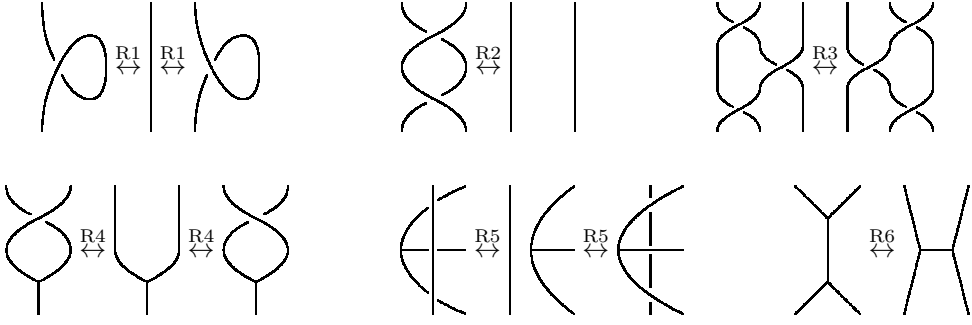
\begin{figure}
\mbox{}\hspace{1em}
\begin{minipage}{24pt}
\begin{picture}(24,48)
 \qbezier(0,0)(0,12)(6,24) 
 \qbezier(6,24)(12,36)(18,36) 
 \qbezier(18,12)(12,12)(7.5,21) 
 \qbezier(4.5,27)(0,36)(0,48) 
 \qbezier(18,12)(24,12)(24,24)
 \qbezier(18,36)(24,36)(24,24)
\end{picture}
\end{minipage}
~$\overset{\text{R1}}{\leftrightarrow}$~
\begin{minipage}{0pt}
\begin{picture}(0,48)
 \qbezier(0,0)(0,24)(0,48)
\end{picture}
\end{minipage}
~$\overset{\text{R1}}{\leftrightarrow}$~
\begin{minipage}{24pt}
\begin{picture}(24,48)
 \qbezier(0,48)(0,36)(6,24) 
 \qbezier(6,24)(12,12)(18,12) 
 \qbezier(18,36)(12,36)(7.5,27) 
 \qbezier(4.5,21)(0,12)(0,0) 
 \qbezier(18,12)(24,12)(24,24)
 \qbezier(18,36)(24,36)(24,24)
\end{picture}
\end{minipage}
\hfill
\begin{minipage}{24pt}
\begin{picture}(24,48)
 \qbezier(0,24)(0,30)(12,36) 
 \qbezier(12,36)(24,42)(24,48) 
 \qbezier(24,24)(24,30)(15,34.5) 
 \qbezier(9,37.5)(0,42)(0,48) 
 \qbezier(24,0)(24,6)(12,12) 
 \qbezier(12,12)(0,18)(0,24) 
 \qbezier(0,0)(0,6)(9,10.5) 
 \qbezier(15,13.5)(24,18)(24,24) 
\end{picture}
\end{minipage}
~$\overset{\text{R2}}{\leftrightarrow}$~
\begin{minipage}{24pt}
\begin{picture}(24,48)
 \qbezier(0,0)(0,24)(0,48)
 \qbezier(24,0)(24,24)(24,48)
\end{picture}
\end{minipage}
\hfill
\begin{minipage}{32pt}
\begin{picture}(32,48)
 \qbezier(0,32)(0,36)(8,40) 
 \qbezier(8,40)(16,44)(16,48) 
 \qbezier(16,32)(16,36)(10,39) 
 \qbezier(6,41)(0,44)(0,48) 
 \qbezier(32,32)(32,40)(32,48)
 \qbezier(16,16)(16,20)(24,24) 
 \qbezier(24,24)(32,28)(32,32) 
 \qbezier(32,16)(32,20)(26,23) 
 \qbezier(22,25)(16,28)(16,32) 
 \qbezier(0,16)(0,24)(0,32)
 \qbezier(0,0)(0,4)(8,8) 
 \qbezier(8,8)(16,12)(16,16) 
 \qbezier(16,0)(16,4)(10,7) 
 \qbezier(6,9)(0,12)(0,16) 
 \qbezier(32,0)(32,8)(32,16)
\end{picture}
\end{minipage}
~$\overset{\text{R3}}{\leftrightarrow}$~
\begin{minipage}{32pt}
\begin{picture}(32,48)
 \qbezier(16,32)(16,36)(24,40) 
 \qbezier(24,40)(32,44)(32,48) 
 \qbezier(32,32)(32,36)(26,39) 
 \qbezier(22,41)(16,44)(16,48) 
 \qbezier(0,32)(0,40)(0,48)
 \qbezier(0,16)(0,20)(8,24) 
 \qbezier(8,24)(16,28)(16,32) 
 \qbezier(16,16)(16,20)(10,23) 
 \qbezier(6,25)(0,28)(0,32) 
 \qbezier(32,16)(32,24)(32,32)
 \qbezier(16,0)(16,4)(24,8) 
 \qbezier(24,8)(32,12)(32,16) 
 \qbezier(32,0)(32,4)(26,7) 
 \qbezier(22,9)(16,12)(16,16) 
 \qbezier(0,0)(0,8)(0,16)
\end{picture}
\end{minipage}
\hspace{1em}\mbox{}\vspace{2em}\\
\begin{minipage}{24pt}
\begin{picture}(24,48)
 \qbezier(0,24)(0,30)(12,36) 
 \qbezier(12,36)(24,42)(24,48) 
 \qbezier(24,24)(24,30)(15,34.5) 
 \qbezier(9,37.5)(0,42)(0,48) 
 \qbezier(12,12)(12,6)(12,0) 
 \qbezier(12,12)(0,18)(0,24) 
 \qbezier(12,12)(24,18)(24,24) 
\end{picture}
\end{minipage}
~$\overset{\text{R4}}{\leftrightarrow}$~
\begin{minipage}{24pt}
\begin{picture}(24,48)
 \qbezier(0,48)(0,36)(0,24)
 \qbezier(24,48)(24,36)(24,24)
 \qbezier(12,12)(12,6)(12,0) 
 \qbezier(12,12)(0,18)(0,24) 
 \qbezier(12,12)(24,18)(24,24) 
\end{picture}
\end{minipage}
~$\overset{\text{R4}}{\leftrightarrow}$~
\begin{minipage}{24pt}
\begin{picture}(24,48)
 \qbezier(24,24)(24,30)(12,36) 
 \qbezier(12,36)(0,42)(0,48) 
 \qbezier(0,24)(0,30)(9,34.5) 
 \qbezier(15,37.5)(24,42)(24,48) 
 \qbezier(12,12)(12,6)(12,0) 
 \qbezier(12,12)(0,18)(0,24) 
 \qbezier(12,12)(24,18)(24,24) 
\end{picture}
\end{minipage}
\hfill
\begin{minipage}{24pt}
\begin{picture}(24,48)
 \qbezier(0,24)(0,32)(10,40)\qbezier(14,43)(19,46)(24,48)
 \qbezier(0,24)(5,24)(10,24)\qbezier(14,24)(19,24)(24,24)
 \qbezier(0,24)(0,16)(10,8)\qbezier(14,5)(19,2)(24,0)
 \qbezier(12,0)(12,24)(12,48)
\end{picture}
\end{minipage}
~$\overset{\text{R5}}{\leftrightarrow}$~
\begin{minipage}{24pt}
\begin{picture}(24,48)
 \qbezier(8,24)(8,36)(24,48)
 \qbezier(8,24)(16,24)(24,24)
 \qbezier(8,24)(8,12)(24,0)
 \qbezier(0,0)(0,24)(0,48)
\end{picture}
\end{minipage}
~$\overset{\text{R5}}{\leftrightarrow}$~
\begin{minipage}{24pt}
\begin{picture}(24,48)
 \qbezier(0,24)(0,36)(24,48)
 \qbezier(0,24)(12,24)(24,24)
 \qbezier(0,24)(0,12)(24,0)
 \qbezier(12,0)(12,2)(12,4)
 \qbezier(12,10)(12,16)(12,22)
 \qbezier(12,26)(12,32)(12,38)
 \qbezier(12,44)(12,46)(12,48)
\end{picture}
\end{minipage}
\hfill
\begin{minipage}{24pt}
\begin{picture}(24,48)
 \qbezier(0,48)(6,42)(12,36)
 \qbezier(24,48)(18,42)(12,36)
 \qbezier(12,12)(12,24)(12,36)
 \qbezier(0,0)(6,6)(12,12)
 \qbezier(24,0)(18,6)(12,12)
\end{picture}
\end{minipage}
~$\overset{\text{R6}}{\leftrightarrow}$~
\begin{minipage}{24pt}
\begin{picture}(24,48)
 \qbezier(0,48)(3,36)(6,24)
 \qbezier(24,48)(21,36)(18,24)
 \qbezier(6,24)(12,24)(18,24)
 \qbezier(0,0)(3,12)(6,24)
 \qbezier(24,0)(21,12)(18,24)
\end{picture}
\end{minipage}
\caption{The Reidemeister moves for handlebody-links}
\label{fig:ReidemeisterMoves}
\end{figure}

For a diagram $D$ of a Y-oriented spatial trivalent graph, we denote by $\mathcal{SA}(D)$ the set of semi-arcs of $D$, where a semi-arc is a piece of a curve each of whose endpoints is a crossing or a vertex.

\begin{definition} \label{def:coloring}
Let $X=\bigsqcup_{\lambda\in\Lambda}G_\lambda$ be a multiple conjugation biquandle.
We define $a\triangle b:=b^{-1}a\oline{*}b$ for $a,b\in G_\lambda$.
Let $D$ be a diagram of an $S^1$-oriented handlebody-link $H$.
An \textit{$X$-coloring} of $D$ is a map $C:\mathcal{SA}(D)\to X$ satisfying
\vspace{1em}
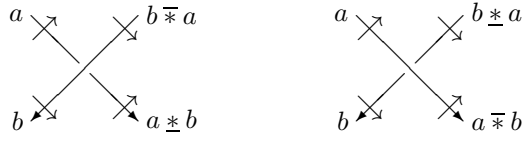
\begin{center}
\begin{minipage}{65pt}
\begin{picture}(65,40)(-5,0)
 \put(40,40){\vector(-1,-1){40}}
 \put(0,40){\line(1,-1){18}}
 \put(22,18){\vector(1,-1){18}}
 \put(5,35){\makebox(0,0){\normalsize$\nearrow$}}
 \put(5,5){\makebox(0,0){\normalsize$\searrow$}}
 \put(35,35){\makebox(0,0){\normalsize$\searrow$}}
 \put(35,5){\makebox(0,0){\normalsize$\nearrow$}}
 \put(-3,40){\makebox(0,0)[r]{\normalsize$a$}}
 \put(-3,0){\makebox(0,0)[r]{\normalsize$b$}}
 \put(43,40){\makebox(0,0)[l]{\normalsize$b\oline{*}a$}}
 \put(43,0){\makebox(0,0)[l]{\normalsize$a\uline{*}b$}}
\end{picture}
\end{minipage}
\hspace{5em}
\begin{minipage}{65pt}
\begin{picture}(65,40)(-5,0)
 \put(0,40){\vector(1,-1){40}}
 \put(40,40){\line(-1,-1){18}}
 \put(18,18){\vector(-1,-1){18}}
 \put(5,35){\makebox(0,0){\normalsize$\nearrow$}}
 \put(5,5){\makebox(0,0){\normalsize$\searrow$}}
 \put(35,35){\makebox(0,0){\normalsize$\searrow$}}
 \put(35,5){\makebox(0,0){\normalsize$\nearrow$}}
 \put(-3,40){\makebox(0,0)[r]{\normalsize$a$}}
 \put(-3,0){\makebox(0,0)[r]{\normalsize$b$}}
 \put(43,40){\makebox(0,0)[l]{\normalsize$b\uline{*}a$}}
 \put(43,0){\makebox(0,0)[l]{\normalsize$a\oline{*}b$}}
\end{picture}
\end{minipage}
\end{center}
\vspace{1em}
at each crossing, and
\vspace{1em}
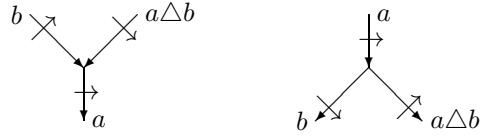
\begin{center}
\begin{minipage}{50pt}
\begin{picture}(50,40)(-5,0)
 \put(20,20){\vector(0,-1){20}}
 \put(0,40){\vector(1,-1){20}}
 \put(40,40){\vector(-1,-1){20}}
 \put(21,10){\makebox(0,0){\normalsize$\rightarrow$}}
 \put(5,35){\makebox(0,0){\normalsize$\nearrow$}}
 \put(35,35){\makebox(0,0){\normalsize$\searrow$}}
 \put(-3,40){\makebox(0,0)[r]{\normalsize$b$}}
 \put(43,40){\makebox(0,0)[l]{\normalsize$a\triangle b$}}
 \put(23,0){\makebox(0,0)[l]{\normalsize$a$}}
\end{picture}
\end{minipage}
\hspace{5em}
\begin{minipage}{50pt}
\begin{picture}(50,40)(-5,0)
 \put(20,40){\vector(0,-1){20}}
 \put(20,20){\vector(-1,-1){20}}
 \put(20,20){\vector(1,-1){20}}
 \put(21,30){\makebox(0,0){\normalsize$\rightarrow$}}
 \put(5,5){\makebox(0,0){\normalsize$\searrow$}}
 \put(35,5){\makebox(0,0){\normalsize$\nearrow$}}
 \put(23,40){\makebox(0,0)[l]{\normalsize$a$}}
 \put(-3,0){\makebox(0,0)[r]{\normalsize$b$}}
 \put(43,0){\makebox(0,0)[l]{\normalsize$a\triangle b$}}
\end{picture}
\end{minipage}
\end{center}
\vspace{1em}
at each vertex, where the normal orientation is obtained by rotating the usual orientation counterclockwise by $\pi/2$ on the diagram.
We denote by $\operatorname{Col}_X(D)$ the set of $X$-colorings of $D$.
\end{definition}

\begin{theorem} \label{thm:coloring}
Let $X=\bigsqcup_{\lambda\in\Lambda}G_\lambda$ be a multiple conjugation biquandle.
Let $D$ be a diagram of an $S^1$-oriented handlebody-link $H$.
Let $D'$ be a diagram obtained by applying one of the Y-oriented R1--R6 moves to the diagram $D$ once.
For an $X$-coloring $C$ of $D$, there is a unique $X$-coloring $C'$ of $D'$ which coincides with $C$ except the place where the move is applied.
\end{theorem}

We prove this theorem in the next section.
Here we introduce the primitive conditions for the proof and the universality discussed in Section~\ref{sect:universality}.

Let $X$ be a biquandle, $\triangle:P\to X$ a map, where $P$ is a subset of $X\times X$.
We write $a\sim b$ if $(a,b)\in P$.
We define an $(X,P,\triangle)$-coloring to be a map $C:\mathcal{SA}(D)\to X$ satisfying the conditions as crossings and vertices as in Definition~\ref{def:coloring}.
The following conditions \eqref{eq:R4-1,primitive}--\eqref{eq:R6-4,primitive}, which we call the \textit{primitive conditions}, are the conditions on $(X,P,\triangle)$ from that we obtain a one-to-one correspondence of $(X,P,\triangle)$-colorings on the Reidemeister moves R4--R6 (see Figure~\ref{fig:coloredReidemeisterMove}, where all arcs are directed from top to bottom, except for the Reidemeister moves R4).
\bigskip

\noindent
(R4) For any $a,b,x\in X$,
\begin{align}
a\sim b,x=a\triangle b
&\Leftrightarrow a\uline{*}b\sim x,(a\uline{*}b)\triangle x=b\oline{*}a, \label{eq:R4-1,primitive} \\
a\sim b,x=a\triangle b
&\Leftrightarrow a\oline{*}b\sim x,(a\oline{*}b)\triangle x=b\uline{*}a. \label{eq:R4-2,primitive}
\end{align}

\noindent
(R5) For any $a,b,x\in X$,
\begin{align}
a\sim b
&\Leftrightarrow a\uline{*}x\sim b\uline{*}x \nonumber \\
&\Rightarrow
(x\oline{*}b)\oline{*}(a\triangle b)=x\oline{*}a,
(a\triangle b)\uline{*}(x\oline{*}b)=(a\uline{*}x)\triangle(b\uline{*}x), \label{eq:R5-1,primitive} \\
a\sim b
&\Leftrightarrow a\oline{*}x\sim b\oline{*}x \nonumber \\
&\Rightarrow
(x\uline{*}b)\uline{*}(a\triangle b)=x\uline{*}a,
(a\triangle b)\oline{*}(x\uline{*}b)=(a\oline{*}x)\triangle(b\oline{*}x). \label{eq:R5-2,primitive}
\end{align}

\noindent
(R6) For any $a,b,c,x\in X$,
\begin{align}
&a\sim b,b\sim c,x=b\triangle c
\Rightarrow a\sim c,a\triangle c\sim x,
(a\triangle c)\triangle x=a\triangle b, \label{eq:R6-1,primitive} \\
&\exists!b\in X\text{ s.t.~}a\sim b,b\sim c,
x=b\triangle c,(a\triangle c)\triangle x=a\triangle b
\Leftarrow a\sim c,a\triangle c\sim x, \label{eq:R6-2,primitive} \\
&a\sim b,a\sim c,x=a\triangle c
\Rightarrow b\sim c,x\sim b\triangle c,
x\triangle(b\triangle c)=a\triangle b, \label{eq:R6-3,primitive} \\
&\exists!a\in X\text{ s.t.~}a\sim b,a\sim c,
x=a\triangle c,x\triangle(b\triangle c)=a\triangle b
\Leftarrow b\sim c,x\sim b\triangle c. \label{eq:R6-4,primitive}
\end{align}

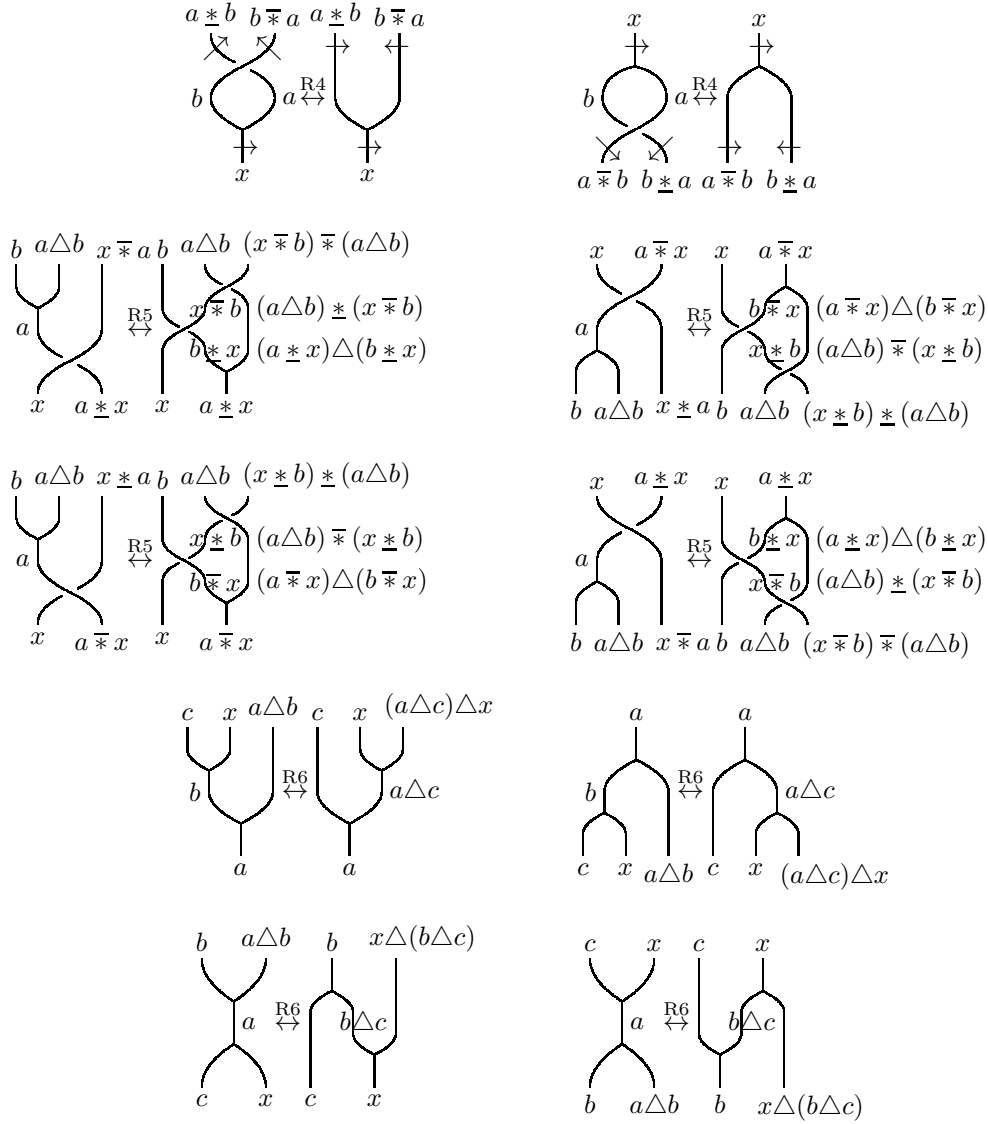
\begin{figure}
\mbox{}\hfill
\begin{minipage}{36pt}
\begin{picture}(24,72)(-6,-12)
 \qbezier(0,24)(0,30)(12,36) 
 \qbezier(12,36)(24,42)(24,48) 
 \qbezier(24,24)(24,30)(15,34.5) 
 \qbezier(9,37.5)(0,42)(0,48) 
 \qbezier(12,12)(12,6)(12,0) 
 \qbezier(12,12)(0,18)(0,24) 
 \qbezier(12,12)(24,18)(24,24) 
 \put(2,43){\makebox(0,0){\normalsize$\nearrow$}}
 \put(22,43){\makebox(0,0){\normalsize$\nwarrow$}}
 \put(13,5){\makebox(0,0){\normalsize$\rightarrow$}}
 \put(0,51){\makebox(0,0)[b]{\normalsize$a\uline{*}b$}}
 \put(24,51){\makebox(0,0)[b]{\normalsize$b\oline{*}a$}}
 \put(12,-3){\makebox(0,0)[t]{\normalsize$x$}}
 \put(-3,24){\makebox(0,0)[r]{\normalsize$b$}}
 \put(27,24){\makebox(0,0)[l]{\normalsize$a$}}
\end{picture}
\end{minipage}
~$\overset{\text{R4}}{\leftrightarrow}$~
\begin{minipage}{24pt}
\begin{picture}(24,72)(0,-12)
 \qbezier(0,48)(0,36)(0,24)
 \qbezier(24,48)(24,36)(24,24)
 \qbezier(12,12)(12,6)(12,0) 
 \qbezier(12,12)(0,18)(0,24) 
 \qbezier(12,12)(24,18)(24,24) 
 \put(1,43){\makebox(0,0){\normalsize$\rightarrow$}}
 \put(23,43){\makebox(0,0){\normalsize$\leftarrow$}}
 \put(13,5){\makebox(0,0){\normalsize$\rightarrow$}}
 \put(0,51){\makebox(0,0)[b]{\normalsize$a\uline{*}b$}}
 \put(24,51){\makebox(0,0)[b]{\normalsize$b\oline{*}a$}}
 \put(12,-3){\makebox(0,0)[t]{\normalsize$x$}}
\end{picture}
\end{minipage}
\hfill
\begin{minipage}{36pt}
\begin{picture}(24,72)(-6,-12)
 \qbezier(12,36)(12,42)(12,48) 
 \qbezier(12,36)(0,33)(0,24) 
 \qbezier(12,36)(24,33)(24,24) 
 \qbezier(0,0)(0,6)(12,12) 
 \qbezier(12,12)(24,18)(24,24) 
 \qbezier(24,0)(24,6)(15,10.5) 
 \qbezier(9,13.5)(0,18)(0,24) 
 \put(13,43){\makebox(0,0){\normalsize$\rightarrow$}}
 \put(2,5){\makebox(0,0){\normalsize$\searrow$}}
 \put(22,5){\makebox(0,0){\normalsize$\swarrow$}}
 \put(12,51){\makebox(0,0)[b]{\normalsize$x$}}
 \put(0,-3){\makebox(0,0)[t]{\normalsize$a\oline{*}b$}}
 \put(24,-3){\makebox(0,0)[t]{\normalsize$b\uline{*}a$}}
 \put(-3,24){\makebox(0,0)[r]{\normalsize$b$}}
 \put(27,24){\makebox(0,0)[l]{\normalsize$a$}}
\end{picture}
\end{minipage}
~$\overset{\text{R4}}{\leftrightarrow}$~
\begin{minipage}{24pt}
\begin{picture}(24,72)(0,-12)
 \qbezier(12,36)(12,42)(12,48) 
 \qbezier(12,36)(0,30)(0,24) 
 \qbezier(12,36)(24,30)(24,24) 
 \qbezier(0,24)(0,12)(0,0)
 \qbezier(24,24)(24,12)(24,0)
 \put(13,43){\makebox(0,0){\normalsize$\rightarrow$}}
 \put(1,5){\makebox(0,0){\normalsize$\rightarrow$}}
 \put(23,5){\makebox(0,0){\normalsize$\leftarrow$}}
 \put(12,51){\makebox(0,0)[b]{\normalsize$x$}}
 \put(0,-3){\makebox(0,0)[t]{\normalsize$a\oline{*}b$}}
 \put(24,-3){\makebox(0,0)[t]{\normalsize$b\uline{*}a$}}
\end{picture}
\end{minipage}
\hfill\mbox{}\vspace{5mm}

\begin{minipage}{38pt}
\begin{picture}(32,72)(0,-12)
 \qbezier(0,40)(0,44)(0,48)
 \qbezier(16,40)(16,44)(16,48)
 \qbezier(8,32)(8,28)(8,24) 
 \qbezier(8,32)(0,36)(0,40) 
 \qbezier(8,32)(16,36)(16,40) 
 \qbezier(32,24)(32,36)(32,48)
 \qbezier(8,0)(8,6)(20,12) 
 \qbezier(20,12)(32,18)(32,24) 
 \qbezier(32,0)(32,6)(23,10.5) 
 \qbezier(17,13.5)(8,18)(8,24) 
 \put(0,51){\makebox(0,0)[b]{\normalsize$b$}}
 \put(16,51){\makebox(0,0)[b]{\normalsize$a\triangle b$}}
 \put(30,51){\makebox(0,0)[bl]{\normalsize$x\oline{*}a$}}
 \put(8,-3){\makebox(0,0)[t]{\normalsize$x$}}
 \put(32,-3){\makebox(0,0)[t]{\normalsize$a\uline{*}x$}}
 \put(5,24){\makebox(0,0)[r]{\normalsize$a$}}
\end{picture}
\end{minipage}
~$\overset{\text{R5}}{\leftrightarrow}$~
\begin{minipage}{96pt}
\begin{picture}(32,72)(0,-12)
 \qbezier(16,32)(16,36)(24,40) 
 \qbezier(24,40)(32,44)(32,48) 
 \qbezier(32,32)(32,36)(26,39) 
 \qbezier(22,41)(16,44)(16,48) 
 \qbezier(0,32)(0,40)(0,48)
 \qbezier(0,16)(0,20)(8,24) 
 \qbezier(8,24)(16,28)(16,32) 
 \qbezier(16,16)(16,20)(10,23) 
 \qbezier(6,25)(0,28)(0,32) 
 \qbezier(32,16)(32,24)(32,32)
 \qbezier(24,8)(24,4)(24,0) 
 \qbezier(24,8)(16,12)(16,16) 
 \qbezier(24,8)(32,12)(32,16) 
 \qbezier(0,0)(0,8)(0,16)
 \put(0,51){\makebox(0,0)[b]{\normalsize$b$}}
 \put(16,51){\makebox(0,0)[b]{\normalsize$a\triangle b$}}
 \put(30,51){\makebox(0,0)[bl]{\normalsize$(x\oline{*}b)\oline{*}(a\triangle b)$}}
 \put(0,-3){\makebox(0,0)[t]{\normalsize$x$}}
 \put(24,-3){\makebox(0,0)[t]{\normalsize$a\uline{*}x$}}
 \put(10,32){\makebox(0,0)[l]{\normalsize$x\oline{*}b$}}
 \put(10,16){\makebox(0,0)[l]{\normalsize$b\uline{*}x$}}
 \put(35,32){\makebox(0,0)[l]{\normalsize$(a\triangle b)\uline{*}(x\oline{*}b)$}}
 \put(35,16){\makebox(0,0)[l]{\normalsize$(a\uline{*}x)\triangle(b\uline{*}x)$}}
\end{picture}
\end{minipage}
\hfill
\begin{minipage}{38pt}
\begin{picture}(32,72)(0,-12)
 \qbezier(8,24)(8,30)(20,36) 
 \qbezier(20,36)(32,42)(32,48) 
 \qbezier(32,24)(32,30)(23,34.5) 
 \qbezier(17,37.5)(8,42)(8,48) 
 \qbezier(8,16)(8,20)(8,24) 
 \qbezier(8,16)(0,12)(0,8) 
 \qbezier(8,16)(16,12)(16,8) 
 \qbezier(0,0)(0,4)(0,8)
 \qbezier(16,0)(16,4)(16,8)
 \qbezier(32,0)(32,12)(32,24)
 \put(8,51){\makebox(0,0)[b]{\normalsize$x$}}
 \put(32,51){\makebox(0,0)[b]{\normalsize$a\oline{*}x$}}
 \put(0,-3){\makebox(0,0)[t]{\normalsize$b$}}
 \put(16,-3){\makebox(0,0)[t]{\normalsize$a\triangle b$}}
 \put(30,-3){\makebox(0,0)[tl]{\normalsize$x\uline{*}a$}}
 \put(5,24){\makebox(0,0)[r]{\normalsize$a$}}
\end{picture}
\end{minipage}
~$\overset{\text{R5}}{\leftrightarrow}$~
\begin{minipage}{96pt}
\begin{picture}(32,72)(0,-12)
 \qbezier(24,40)(24,44)(24,48) 
 \qbezier(24,40)(16,36)(16,32) 
 \qbezier(24,40)(32,36)(32,32) 
 \qbezier(0,32)(0,40)(0,48)
 \qbezier(0,16)(0,20)(8,24) 
 \qbezier(8,24)(16,28)(16,32) 
 \qbezier(16,16)(16,20)(10,23) 
 \qbezier(6,25)(0,28)(0,32) 
 \qbezier(32,16)(32,24)(32,32)
 \qbezier(16,0)(16,4)(24,8) 
 \qbezier(24,8)(32,12)(32,16) 
 \qbezier(32,0)(32,4)(26,7) 
 \qbezier(22,9)(16,12)(16,16) 
 \qbezier(0,0)(0,8)(0,16)
 \put(0,51){\makebox(0,0)[b]{\normalsize$x$}}
 \put(24,51){\makebox(0,0)[b]{\normalsize$a\oline{*}x$}}
 \put(0,-3){\makebox(0,0)[t]{\normalsize$b$}}
 \put(16,-3){\makebox(0,0)[t]{\normalsize$a\triangle b$}}
 \put(30,-3){\makebox(0,0)[tl]{\normalsize$(x\uline{*}b)\uline{*}(a\triangle b)$}}
 \put(10,32){\makebox(0,0)[l]{\normalsize$b\oline{*}x$}}
 \put(10,16){\makebox(0,0)[l]{\normalsize$x\uline{*}b$}}
 \put(35,32){\makebox(0,0)[l]{\normalsize$(a\oline{*}x)\triangle(b\oline{*}x)$}}
 \put(35,16){\makebox(0,0)[l]{\normalsize$(a\triangle b)\oline{*}(x\uline{*}b)$}}
\end{picture}
\end{minipage}\vspace{5mm}

\begin{minipage}{38pt}
\begin{picture}(32,72)(0,-12)
 \qbezier(0,40)(0,44)(0,48)
 \qbezier(16,40)(16,44)(16,48)
 \qbezier(8,32)(8,28)(8,24) 
 \qbezier(8,32)(0,36)(0,40) 
 \qbezier(8,32)(16,36)(16,40) 
 \qbezier(32,24)(32,36)(32,48)
 \qbezier(32,0)(32,6)(20,12) 
 \qbezier(20,12)(8,18)(8,24) 
 \qbezier(8,0)(8,6)(17,10.5) 
 \qbezier(23,13.5)(32,18)(32,24) 
 \put(0,51){\makebox(0,0)[b]{\normalsize$b$}}
 \put(16,51){\makebox(0,0)[b]{\normalsize$a\triangle b$}}
 \put(30,51){\makebox(0,0)[bl]{\normalsize$x\uline{*}a$}}
 \put(8,-3){\makebox(0,0)[t]{\normalsize$x$}}
 \put(32,-3){\makebox(0,0)[t]{\normalsize$a\oline{*}x$}}
 \put(5,24){\makebox(0,0)[r]{\normalsize$a$}}
\end{picture}
\end{minipage}
~$\overset{\text{R5}}{\leftrightarrow}$~
\begin{minipage}{96pt}
\begin{picture}(32,72)(0,-12)
 \qbezier(32,32)(32,36)(24,40) 
 \qbezier(24,40)(16,44)(16,48) 
 \qbezier(16,32)(16,36)(22,39) 
 \qbezier(26,41)(32,44)(32,48) 
 \qbezier(0,32)(0,40)(0,48)
 \qbezier(16,16)(16,20)(8,24) 
 \qbezier(8,24)(0,28)(0,32) 
 \qbezier(0,16)(0,20)(6,23) 
 \qbezier(10,25)(16,28)(16,32) 
 \qbezier(32,16)(32,24)(32,32)
 \qbezier(24,8)(24,4)(24,0) 
 \qbezier(24,8)(16,12)(16,16) 
 \qbezier(24,8)(32,12)(32,16) 
 \qbezier(0,0)(0,8)(0,16)
 \put(0,51){\makebox(0,0)[b]{\normalsize$b$}}
 \put(16,51){\makebox(0,0)[b]{\normalsize$a\triangle b$}}
 \put(30,51){\makebox(0,0)[bl]{\normalsize$(x\uline{*}b)\uline{*}(a\triangle b)$}}
 \put(0,-3){\makebox(0,0)[t]{\normalsize$x$}}
 \put(24,-3){\makebox(0,0)[t]{\normalsize$a\oline{*}x$}}
 \put(10,32){\makebox(0,0)[l]{\normalsize$x\uline{*}b$}}
 \put(10,16){\makebox(0,0)[l]{\normalsize$b\oline{*}x$}}
 \put(35,32){\makebox(0,0)[l]{\normalsize$(a\triangle b)\oline{*}(x\uline{*}b)$}}
 \put(35,16){\makebox(0,0)[l]{\normalsize$(a\oline{*}x)\triangle(b\oline{*}x)$}}
\end{picture}
\end{minipage}
\hfill
\begin{minipage}{38pt}
\begin{picture}(32,72)(0,-12)
 \qbezier(32,24)(32,30)(20,36) 
 \qbezier(20,36)(8,42)(8,48) 
 \qbezier(8,24)(8,30)(17,34.5) 
 \qbezier(23,37.5)(32,42)(32,48) 
 \qbezier(8,16)(8,20)(8,24) 
 \qbezier(8,16)(0,12)(0,8) 
 \qbezier(8,16)(16,12)(16,8) 
 \qbezier(0,0)(0,4)(0,8)
 \qbezier(16,0)(16,4)(16,8)
 \qbezier(32,0)(32,12)(32,24)
 \put(8,51){\makebox(0,0)[b]{\normalsize$x$}}
 \put(32,51){\makebox(0,0)[b]{\normalsize$a\uline{*}x$}}
 \put(0,-3){\makebox(0,0)[t]{\normalsize$b$}}
 \put(16,-3){\makebox(0,0)[t]{\normalsize$a\triangle b$}}
 \put(30,-3){\makebox(0,0)[tl]{\normalsize$x\oline{*}a$}}
 \put(5,24){\makebox(0,0)[r]{\normalsize$a$}}
\end{picture}
\end{minipage}
~$\overset{\text{R5}}{\leftrightarrow}$~
\begin{minipage}{96pt}
\begin{picture}(32,72)(0,-12)
 \qbezier(24,40)(24,44)(24,48) 
 \qbezier(24,40)(16,36)(16,32) 
 \qbezier(24,40)(32,36)(32,32) 
 \qbezier(0,32)(0,40)(0,48)
 \qbezier(16,16)(16,20)(8,24) 
 \qbezier(8,24)(0,28)(0,32) 
 \qbezier(0,16)(0,20)(6,23) 
 \qbezier(10,25)(16,28)(16,32) 
 \qbezier(32,16)(32,24)(32,32)
 \qbezier(32,0)(32,4)(24,8) 
 \qbezier(24,8)(16,12)(16,16) 
 \qbezier(16,0)(16,4)(22,7) 
 \qbezier(26,9)(32,12)(32,16) 
 \qbezier(0,0)(0,8)(0,16)
 \put(0,51){\makebox(0,0)[b]{\normalsize$x$}}
 \put(24,51){\makebox(0,0)[b]{\normalsize$a\uline{*}x$}}
 \put(0,-3){\makebox(0,0)[t]{\normalsize$b$}}
 \put(16,-3){\makebox(0,0)[t]{\normalsize$a\triangle b$}}
 \put(30,-3){\makebox(0,0)[tl]{\normalsize$(x\oline{*}b)\oline{*}(a\triangle b)$}}
 \put(10,32){\makebox(0,0)[l]{\normalsize$b\uline{*}x$}}
 \put(10,16){\makebox(0,0)[l]{\normalsize$x\oline{*}b$}}
 \put(35,32){\makebox(0,0)[l]{\normalsize$(a\uline{*}x)\triangle(b\uline{*}x)$}}
 \put(35,16){\makebox(0,0)[l]{\normalsize$(a\triangle b)\uline{*}(x\oline{*}b)$}}
\end{picture}
\end{minipage}\vspace{5mm}

\mbox{}\hfill
\begin{minipage}{32pt}
\begin{picture}(32,72)(0,-12)
 \qbezier(0,40)(0,44)(0,48)
 \qbezier(16,40)(16,44)(16,48)
 \qbezier(8,32)(8,28)(8,24) 
 \qbezier(8,32)(0,36)(0,40) 
 \qbezier(8,32)(16,36)(16,40) 
 \qbezier(32,24)(32,36)(32,48)
 \qbezier(20,12)(20,6)(20,0) 
 \qbezier(20,12)(8,18)(8,24) 
 \qbezier(20,12)(32,18)(32,24) 
 \put(0,51){\makebox(0,0)[b]{\normalsize$c$}}
 \put(16,51){\makebox(0,0)[b]{\normalsize$x$}}
 \put(32,51){\makebox(0,0)[b]{\normalsize$a\triangle b$}}
 \put(20,-3){\makebox(0,0)[t]{\normalsize$a$}}
 \put(5,24){\makebox(0,0)[r]{\normalsize$b$}}
\end{picture}
\end{minipage}
~$\overset{\text{R6}}{\leftrightarrow}$~
\begin{minipage}{32pt}
\begin{picture}(32,72)(0,-12)
 \qbezier(16,40)(16,44)(16,48)
 \qbezier(32,40)(32,44)(32,48)
 \qbezier(24,32)(24,28)(24,24) 
 \qbezier(24,32)(16,36)(16,40) 
 \qbezier(24,32)(32,36)(32,40) 
 \qbezier(0,24)(0,36)(0,48)
 \qbezier(12,12)(12,6)(12,0) 
 \qbezier(12,12)(0,18)(0,24) 
 \qbezier(12,12)(24,18)(24,24) 
 \put(0,51){\makebox(0,0)[b]{\normalsize$c$}}
 \put(16,51){\makebox(0,0)[b]{\normalsize$x$}}
 \put(25,51){\makebox(0,0)[bl]{\normalsize$(a\triangle c)\triangle x$}}
 \put(12,-3){\makebox(0,0)[t]{\normalsize$a$}}
 \put(27,24){\makebox(0,0)[l]{\normalsize$a\triangle c$}}
\end{picture}
\end{minipage}
\hfill
\begin{minipage}{32pt}
\begin{picture}(32,72)(0,-12)
 \qbezier(20,36)(20,42)(20,48) 
 \qbezier(20,36)(8,30)(8,24) 
 \qbezier(20,36)(32,30)(32,24) 
 \qbezier(32,0)(32,12)(32,24)
 \qbezier(8,16)(8,20)(8,24) 
 \qbezier(8,16)(0,12)(0,8) 
 \qbezier(8,16)(16,12)(16,8) 
 \qbezier(0,0)(0,4)(0,8)
 \qbezier(16,0)(16,4)(16,8)
 \put(20,51){\makebox(0,0)[b]{\normalsize$a$}}
 \put(0,-3){\makebox(0,0)[t]{\normalsize$c$}}
 \put(16,-3){\makebox(0,0)[t]{\normalsize$x$}}
 \put(32,-3){\makebox(0,0)[t]{\normalsize$a\triangle b$}}
 \put(5,24){\makebox(0,0)[r]{\normalsize$b$}}
\end{picture}
\end{minipage}
~$\overset{\text{R6}}{\leftrightarrow}$~
\begin{minipage}{32pt}
\begin{picture}(32,72)(0,-12)
 \qbezier(12,36)(12,42)(12,48) 
 \qbezier(12,36)(0,30)(0,24) 
 \qbezier(12,36)(24,30)(24,24) 
 \qbezier(0,0)(0,12)(0,24)
 \qbezier(24,16)(24,20)(24,24) 
 \qbezier(24,16)(16,12)(16,8) 
 \qbezier(24,16)(32,12)(32,8) 
 \qbezier(16,0)(16,4)(16,8)
 \qbezier(32,0)(32,4)(32,8)
 \put(12,51){\makebox(0,0)[b]{\normalsize$a$}}
 \put(0,-3){\makebox(0,0)[t]{\normalsize$c$}}
 \put(16,-3){\makebox(0,0)[t]{\normalsize$x$}}
 \put(25,-3){\makebox(0,0)[tl]{\normalsize$(a\triangle c)\triangle x$}}
 \put(27,24){\makebox(0,0)[l]{\normalsize$a\triangle c$}}
\end{picture}
\end{minipage}
\hfill\mbox{}\vspace{5mm}

\mbox{}\hfill
\begin{minipage}{24pt}
\begin{picture}(24,72)(0,-12)
 \qbezier(12,32)(12,28)(12,24) 
 \qbezier(12,32)(0,40)(0,48) 
 \qbezier(12,32)(24,40)(24,48) 
 \qbezier(12,16)(12,20)(12,24) 
 \qbezier(12,16)(0,8)(0,0) 
 \qbezier(12,16)(24,8)(24,0) 
 \put(0,51){\makebox(0,0)[b]{\normalsize$b$}}
 \put(24,51){\makebox(0,0)[b]{\normalsize$a\triangle b$}}
 \put(0,-3){\makebox(0,0)[t]{\normalsize$c$}}
 \put(24,-3){\makebox(0,0)[t]{\normalsize$x$}}
 \put(15,24){\makebox(0,0)[l]{\normalsize$a$}}
\end{picture}
\end{minipage}
~$\overset{\text{R6}}{\leftrightarrow}$~
\begin{minipage}{32pt}
\begin{picture}(32,72)(0,-12)
 \qbezier(8,36)(8,42)(8,48) 
 \qbezier(8,36)(0,32)(0,28) 
 \qbezier(8,36)(16,32)(16,28) 
 \qbezier(0,0)(0,14)(0,28)
 \qbezier(16,20)(16,24)(16,28)
 \qbezier(32,20)(32,34)(32,48)
 \qbezier(24,12)(24,6)(24,0) 
 \qbezier(24,12)(16,16)(16,20) 
 \qbezier(24,12)(32,16)(32,20) 
 \put(8,51){\makebox(0,0)[b]{\normalsize$b$}}
 \put(22,51){\makebox(0,0)[bl]{\normalsize$x\triangle(b\triangle c)$}}
 \put(0,-3){\makebox(0,0)[t]{\normalsize$c$}}
 \put(24,-3){\makebox(0,0)[t]{\normalsize$x$}}
 \put(11,24){\makebox(0,0)[l]{\normalsize$b\triangle c$}}
\end{picture}
\end{minipage}
\hfill
\begin{minipage}{24pt}
\begin{picture}(24,72)(0,-12)
 \qbezier(12,32)(12,28)(12,24) 
 \qbezier(12,32)(0,40)(0,48) 
 \qbezier(12,32)(24,40)(24,48) 
 \qbezier(12,16)(12,20)(12,24) 
 \qbezier(12,16)(0,8)(0,0) 
 \qbezier(12,16)(24,8)(24,0) 
 \put(0,51){\makebox(0,0)[b]{\normalsize$c$}}
 \put(24,51){\makebox(0,0)[b]{\normalsize$x$}}
 \put(0,-3){\makebox(0,0)[t]{\normalsize$b$}}
 \put(24,-3){\makebox(0,0)[t]{\normalsize$a\triangle b$}}
 \put(15,24){\makebox(0,0)[l]{\normalsize$a$}}
\end{picture}
\end{minipage}
~$\overset{\text{R6}}{\leftrightarrow}$~
\begin{minipage}{32pt}
\begin{picture}(32,72)(0,-12)
 \qbezier(24,36)(24,42)(24,48) 
 \qbezier(24,36)(16,32)(16,28) 
 \qbezier(24,36)(32,32)(32,28) 
 \qbezier(0,20)(0,34)(0,48)
 \qbezier(16,20)(16,24)(16,28)
 \qbezier(32,0)(32,14)(32,28)
 \qbezier(8,12)(8,6)(8,0) 
 \qbezier(8,12)(0,16)(0,20) 
 \qbezier(8,12)(16,16)(16,20) 
 \put(0,51){\makebox(0,0)[b]{\normalsize$c$}}
 \put(24,51){\makebox(0,0)[b]{\normalsize$x$}}
 \put(8,-3){\makebox(0,0)[t]{\normalsize$b$}}
 \put(22,-3){\makebox(0,0)[tl]{\normalsize$x\triangle(b\triangle c)$}}
 \put(11,24){\makebox(0,0)[l]{\normalsize$b\triangle c$}}
\end{picture}
\end{minipage}
\hfill\mbox{}\vspace{5mm}
\caption{Colored Reidemeister moves}
\label{fig:coloredReidemeisterMove}
\end{figure}

\section{Proof of Theorem~\ref{thm:coloring}}

\begin{lemma} \label{lem:axioms of triangle MCB}
Let $X=\bigsqcup_{\lambda\in\Lambda}G_\lambda$ be a multiple conjugation biquandle with $a\triangle b:=b^{-1}a\oline{*}b$.
We have the following.
\begin{itemize}
\item
For any $a\in X$,
\begin{align}
\text{$\triangle a:G_a\to G_{a\triangle a}$ which sends $x$ to $x\triangle a$ is a bijection.} \label{eq:bijection,triangle}
\end{align}
\item
For any $a,x\in X$,
\begin{align}
\text{$\uline{*}x:G_a\to G_{a\uline{*}x}$ and $\oline{*}x:G_a\to G_{a\oline{*}x}$ are bijections.} \label{eq:bijection,uo}
\end{align}
\item
For any $a,b\in G_\lambda$,
\begin{align}
&G_{a\uline{*}b}=G_{a\triangle b},
&&(a\uline{*}b)\triangle(a\triangle b)=b\oline{*}a, \label{eq:R4,u,triangle} \\
&G_{a\oline{*}b}=G_{a\triangle b},
&&(a\oline{*}b)\triangle(a\triangle b)=b\uline{*}a. \label{eq:R4,o,triangle}
\end{align}
\item
For any $a,b\in G_\lambda$ and $x\in X$,
\begin{align}
&(a\triangle b)\uline{*}(x\oline{*}b)=(a\uline{*}x)\triangle(b\uline{*}x), \label{eq:R5-1,u,triangle} \\
&(a\triangle b)\oline{*}(x\uline{*}b)=(a\oline{*}x)\triangle(b\oline{*}x), \label{eq:R5-1,o,triangle} \\
&(x\uline{*}b)\uline{*}(a\triangle b)=x\uline{*}a, \label{eq:R5-2,u,triangle} \\
&(x\oline{*}b)\oline{*}(a\triangle b)=x\oline{*}a. \label{eq:R5-2,o,triangle}
\end{align}
\item
For any $a,b,c\in G_\lambda$,
\begin{align}
(a\triangle c)\triangle(b\triangle c)=a\triangle b. \label{eq:R6,triangle}
\end{align}
\end{itemize}
\end{lemma}

\begin{proof}
\begin{itemize}
\item
The map $\triangle a:G_a\to G_{a\triangle a}$ is a well-defined bijection, since it is the composition of the bijections
$a^{-1}\cdot:G_a\to G_a$ defined by $a^{-1}\cdot x=a^{-1}x$ and
$\oline{*}a:G_a\to G_{a\oline{*}a}=G_{a\triangle a}$.

\item
By Lemma~\ref{lem:x*e=x}, $\uline{*}x:G_a\to G_{a\uline{*}x}$ and $\oline{*}x:G_a\to G_{a\oline{*}x}$ are well-defined bijections.

\item
For $a,b\in G_\lambda$, we have $G_{a\uline{*}b}=G_{a\triangle b}=G_{a\oline{*}b}$, since
\begin{align*}
&ab^{-1}\in G_a, && a\triangle b\overset{\eqref{eq:R14}}{=}ab^{-1}\uline{*}b\in G_{a\uline{*}b}, \\
&b^{-1}a\in G_a, && a\triangle b=b^{-1}a\oline{*}b\in G_{a\oline{*}b}.
\end{align*}
For $a,b\in G_\lambda$, we have
\begin{align*}
(a\uline{*}b)\triangle(a\triangle b)
&\overset{\eqref{eq:R14}}{=}
(b^{-1}ab\oline{*}b)\triangle(b^{-1}a\oline{*}b) \\
&=(b^{-1}a\oline{*}b)^{-1}(b^{-1}ab\oline{*}b)\oline{*}(b^{-1}a\oline{*}b) \\
&=(b\oline{*}b)\oline{*}(b^{-1}a\oline{*}b)
\overset{\eqref{eq:x*o(ab)}}{=}b\oline{*}a, \\
(a\oline{*}b)\triangle(a\triangle b)
&=(b^{-1}a\oline{*}b)^{-1}(a\oline{*}b)\oline{*}(b^{-1}a\oline{*}b) \\
&=(a^{-1}ba\oline{*}b)\oline{*}(b^{-1}a\oline{*}b)
\overset{\eqref{eq:x*o(ab)}}{=}a^{-1}ba\oline{*}a
\overset{\eqref{eq:R14}}{=}b\uline{*}a.
\end{align*}

\item
For $a,b\in G_\lambda$ and $x\in X$, we have
\begin{align*}
&(a\triangle b)\uline{*}(x\oline{*}b)
=(b^{-1}a\oline{*}b)\uline{*}(x\oline{*}b)
\overset{\rm(B3)}{=}(b^{-1}a\uline{*}x)\oline{*}(b\uline{*}x) \\
&=(b\uline{*}x)^{-1}(a\uline{*}x)\oline{*}(b\uline{*}x)
=(a\uline{*}x)\triangle(b\uline{*}x), \\
&(a\triangle b)\oline{*}(x\uline{*}b)
=(b^{-1}a\oline{*}b)\oline{*}(x\uline{*}b)
\overset{\rm(B3)}{=}(b^{-1}a\oline{*}x)\oline{*}(b\oline{*}x) \\
&=(b\oline{*}x)^{-1}(a\oline{*}x)\oline{*}(b\oline{*}x)
=(a\oline{*}x)\triangle(b\oline{*}x).
\end{align*}

\item
For $a,b\in G_\lambda$ and $x\in X$, we have
\begin{align*}
&(x\uline{*}b)\uline{*}(a\triangle b)
=(x\uline{*}b)\uline{*}(b^{-1}a\oline{*}b)
\overset{\eqref{eq:x*u(ab)}}{=}x\uline{*}a, \\
&(x\oline{*}b)\oline{*}(a\triangle b)
=(x\oline{*}b)\oline{*}(b^{-1}a\oline{*}b)
\overset{\eqref{eq:x*o(ab)}}{=}x\oline{*}a.
\end{align*}

\item
For $a,b,c\in G_\lambda$, we have
\begin{align*}
(a\triangle c)\triangle(b\triangle c)
&=(c^{-1}a\oline{*}c)\triangle(c^{-1}b\oline{*}c) \\
&=(c^{-1}b\oline{*}c)^{-1}(c^{-1}a\oline{*}c)\oline{*}(c^{-1}b\oline{*}c) \\
&=(b^{-1}a\oline{*}c)\oline{*}(c^{-1}b\oline{*}c) \\
&\overset{\eqref{eq:x*o(ab)}}{=}b^{-1}a\oline{*}b
=a\triangle b.
\end{align*}
\end{itemize}
\end{proof}

\begin{proof}[Proof of Theorem~\ref{thm:coloring}]
We see that $(X,\bigsqcup_{\lambda\in\Lambda}G_\lambda^2,\triangle)$ satisfies the primitive conditions \eqref{eq:R4-1,primitive}--\eqref{eq:R6-4,primitive}.
By Lemma~\ref{lem:axioms of triangle MCB}, it is sufficient to show
\begin{align}
b\in G_a,x=a\triangle b
&\Leftarrow x\in G_{a\uline{*}b},(a\uline{*}b)\triangle x=b\oline{*}a,
\label{eq:R4-1,primitive,half} \\
b\in G_a,x=a\triangle b
&\Leftarrow x\in G_{a\oline{*}b},(a\oline{*}b)\triangle x=b\uline{*}a
\label{eq:R4-2,primitive,half}
\end{align}
for $a,b,c,x\in X$.
The other conditions are easily verified, where we note that $b=x\triangle^{-1}c\sim a$ and $a=x\triangle^{-1}c\sim b$ for \eqref{eq:R6-2,primitive} and \eqref{eq:R6-4,primitive}, respectively.

We show \eqref{eq:R4-1,primitive,half}.
Put $c:=x\uline{*}^{-1}b\in G_a$.
Then
\[ (a\triangle c)\uline{*}(b\oline{*}c)
\overset{\eqref{eq:R5-1,u,triangle}}{=}
(a\uline{*}b)\triangle(c\uline{*}b)
=(a\uline{*}b)\triangle x
=b\oline{*}a
\overset{\eqref{eq:R5-2,o,triangle}}{=}
(b\oline{*}c)\oline{*}(a\triangle c). \]
By Lemma~\ref{lem:x*x=y}, we have $a\triangle c=b\oline{*}c$.
Since $a\triangle c\in G_{a\triangle c}=G_{a\oline{*}c}$, we have $b=(a\triangle c)\oline{*}^{-1}c\in G_a$.
The equality $x=a\triangle b$ follows from
\[ x\triangle(c\triangle b)
=(c\uline{*}b)\triangle(c\triangle b)
\overset{\eqref{eq:R4,u,triangle}}{=}b\oline{*}c
=a\triangle c
\overset{\eqref{eq:R6,triangle}}{=}(a\triangle b)\triangle(c\triangle b). \]
Then we have \eqref{eq:R4-1,primitive,half}.
We show \eqref{eq:R4-2,primitive,half}.
Put $c:=x\oline{*}^{-1}b\in G_a$.
Then
\[ (a\triangle c)\oline{*}(b\uline{*}c)
\overset{\eqref{eq:R5-1,o,triangle}}{=}
(a\oline{*}b)\triangle(c\oline{*}b)
=(a\oline{*}b)\triangle x
=b\uline{*}a
\overset{\eqref{eq:R5-2,u,triangle}}{=}
(b\uline{*}c)\uline{*}(a\triangle c). \]
By Lemma~\ref{lem:x*x=y}, we have $a\triangle c=b\uline{*}c$.
Since $a\triangle c\in G_{a\triangle c}=G_{a\uline{*}c}$, we have $b=(a\triangle c)\uline{*}^{-1}c\in G_a$.
The equality $x=a\triangle b$ follows from
\[ x\triangle(c\triangle b)
=(c\oline{*}b)\triangle(c\triangle b)
\overset{\eqref{eq:R4,o,triangle}}{=}b\uline{*}c
=a\triangle c
\overset{\eqref{eq:R6,triangle}}{=}(a\triangle b)\triangle(c\triangle b). \]
Then we have \eqref{eq:R4-2,primitive,half}.
\end{proof}

\section{The universality of an MCB}
\label{sect:universality}

In this section, we see that a multiple conjugation biquandle is the universal biquandle to define coloring invariants for $S^1$-oriented handlebody-links.

\begin{theorem} \label{thm:universality}
Let $X$ be a biquandle, $\triangle:P\to X$ a map, where $P$ is a subset of $X\times X$.
We write $a\sim b$ if $(a,b)\in P$.
Suppose $(X,P,\triangle)$ satisfies the primitive conditions \eqref{eq:R4-1,primitive}--\eqref{eq:R6-4,primitive}.
\begin{itemize}
\item[(1)]
We define
$X_1:=\{b\in X\,|\,\text{there exists $a\in X$ such that $a\sim b$}\}$,
$X_2:=X-X_1$.
Then $X_1,X_2$ are subbiquandles of $X$ satisfying
\begin{align*}
&X_1\uline{*}a=X_1\oline{*}a=X_1, &&X_2\uline{*}a=X_2\oline{*}a=X_2
\end{align*}
for any $a\in X$, where
$X_i\uline{*}a=\{x\uline{*}a\,|\,x\in X_i\}$,
$X_i\oline{*}a=\{x\oline{*}a\,|\,x\in X_i\}$.
\item[(2)]
The relation $\sim$ is an equivalence relation on $X_1$.
\item[(3)]
Let $X_1=\bigsqcup_{\lambda\in\Lambda}G_\lambda$ be the partition of $X_1$ determined by the equivalence relation $\sim$, that is, $a\sim b$ if and only if $a,b\in G_\lambda$ for some $\lambda\in\Lambda$.
Then $X_1$ is a multiple conjugation biquandle.
\end{itemize}
\end{theorem}

By the definition, elements in $X_2$ cannot be used for colorings at a vertex.
For a handlebody-knot of genus greater than one, we see that they also cannot be used for colors of any arcs.
In this sense, an MCB is the universal biquandle for $S^1$-oriented handlebody-links.
A multiple conjugation quandle (MCQ)~\cite{Ishii15MCQ} was introduced as the universal symmetric quandle for unoriented handlebody-links in the same sense, where we note that the axioms of an MCQ coincide with that of an MCB under the assumption that $x\oline{*}y=x$.

In~\cite{Iijima17}, Iijima showed that an MCQ is also the universal quandle for $S^1$-oriented handlebody-links, although it was introduced as the universal symmetric quandle for unoriented handlebody-links.
As a corollary of Theorem~\ref{thm:universality}, we also have this universality.
In~\cite{IshiiNelson16}, Nelson and the first author introduced the notion of a partially multiplicative biquandle.

\begin{definition}[\cite{IshiiNelson16}]
A \textit{partially multiplicative biquandle} is a biquandle $X$ with a subset $\widetilde{P}$ of $X\times X$ and a map $\bullet:\widetilde{P}\to X$ satisfying the following axioms, where $a\bullet b$ stands for $\bullet(a,b)$.
\begin{itemize}
\item[(i)]
$x\mapsto a\bullet x$, $x\mapsto x\bullet b$ are injective.
\item[(ii)]
$(a,b\uline{*}a)\in\widetilde{P}
\Leftrightarrow(b,a\oline{*}b)\in\widetilde{P}
\Rightarrow a\bullet(b\uline{*}a)=b\bullet(a\oline{*}b)$.
\item[(iii)]
$(a,b)\in\widetilde{P}
\Leftrightarrow(a\uline{*}x,b\uline{*}(x\oline{*}a))\in\widetilde{P}
\Leftrightarrow(a\oline{*}x,b\oline{*}(x\uline{*}a))\in\widetilde{P}
\Rightarrow$
\begin{align*}
&x\uline{*}(a\bullet b)=(x\uline{*}a)\uline{*}b,
&&(a\bullet b)\uline{*}x=(a\uline{*}x)\bullet(b\uline{*}(x\oline{*}a)), \\
&x\oline{*}(a\bullet b)=(x\oline{*}a)\oline{*}b,
&&(a\bullet b)\oline{*}x=(a\oline{*}x)\bullet(b\oline{*}(x\uline{*}a)).
\end{align*}
\item[(iv)]
$(a,b),(a\bullet b,c)\in\widetilde{P}
\Leftrightarrow(b,c),(a,b\bullet c)\in\widetilde{P}
\Rightarrow(a\bullet b)\bullet c=a\bullet(b\bullet c)$.
\item[(v)]
$(a,b),(c,d)\in\widetilde{P},a\bullet b=c\bullet d
\Leftrightarrow\exists e\in X$ such that $(a,e),(e,d)\in\widetilde{P},a\bullet e=c,e\bullet d=b$.
\end{itemize}
\end{definition}

The axioms of a partially multiplicative biquandle is obtained from colored Reidemeister moves like the primitive conditions \eqref{eq:R4-1,primitive}--\eqref{eq:R6-4,primitive}, where the coloring is defined by
\vspace{0.5em}
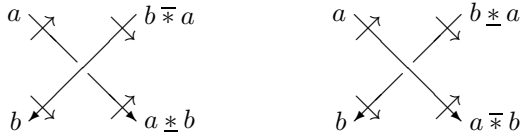
\begin{center}
\begin{minipage}{65pt}
\begin{picture}(65,40)(-5,0)
 \put(40,40){\vector(-1,-1){40}}
 \put(0,40){\line(1,-1){18}}
 \put(22,18){\vector(1,-1){18}}
 \put(5,35){\makebox(0,0){\normalsize$\nearrow$}}
 \put(5,5){\makebox(0,0){\normalsize$\searrow$}}
 \put(35,35){\makebox(0,0){\normalsize$\searrow$}}
 \put(35,5){\makebox(0,0){\normalsize$\nearrow$}}
 \put(-3,40){\makebox(0,0)[r]{\normalsize$a$}}
 \put(-3,0){\makebox(0,0)[r]{\normalsize$b$}}
 \put(43,40){\makebox(0,0)[l]{\normalsize$b\oline{*}a$}}
 \put(43,0){\makebox(0,0)[l]{\normalsize$a\uline{*}b$}}
\end{picture}
\end{minipage}
\hspace{5em}
\begin{minipage}{65pt}
\begin{picture}(65,40)(-5,0)
 \put(0,40){\vector(1,-1){40}}
 \put(40,40){\line(-1,-1){18}}
 \put(18,18){\vector(-1,-1){18}}
 \put(5,35){\makebox(0,0){\normalsize$\nearrow$}}
 \put(5,5){\makebox(0,0){\normalsize$\searrow$}}
 \put(35,35){\makebox(0,0){\normalsize$\searrow$}}
 \put(35,5){\makebox(0,0){\normalsize$\nearrow$}}
 \put(-3,40){\makebox(0,0)[r]{\normalsize$a$}}
 \put(-3,0){\makebox(0,0)[r]{\normalsize$b$}}
 \put(43,40){\makebox(0,0)[l]{\normalsize$b\uline{*}a$}}
 \put(43,0){\makebox(0,0)[l]{\normalsize$a\oline{*}b$}}
\end{picture}
\end{minipage}
\end{center}
\vspace{0.5em}
at each crossing, and
\vspace{0.5em}
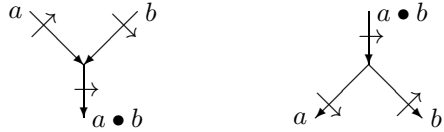
\begin{center}
\begin{minipage}{50pt}
\begin{picture}(50,40)(-5,0)
 \put(20,20){\vector(0,-1){20}}
 \put(0,40){\vector(1,-1){20}}
 \put(40,40){\vector(-1,-1){20}}
 \put(21,10){\makebox(0,0){\normalsize$\rightarrow$}}
 \put(5,35){\makebox(0,0){\normalsize$\nearrow$}}
 \put(35,35){\makebox(0,0){\normalsize$\searrow$}}
 \put(-3,40){\makebox(0,0)[r]{\normalsize$a$}}
 \put(43,40){\makebox(0,0)[l]{\normalsize$b$}}
 \put(23,0){\makebox(0,0)[l]{\normalsize$a\bullet b$}}
\end{picture}
\end{minipage}
\hspace{5em}
\begin{minipage}{50pt}
\begin{picture}(50,40)(-5,0)
 \put(20,40){\vector(0,-1){20}}
 \put(20,20){\vector(-1,-1){20}}
 \put(20,20){\vector(1,-1){20}}
 \put(21,30){\makebox(0,0){\normalsize$\rightarrow$}}
 \put(5,5){\makebox(0,0){\normalsize$\searrow$}}
 \put(35,5){\makebox(0,0){\normalsize$\nearrow$}}
 \put(23,40){\makebox(0,0)[l]{\normalsize$a\bullet b$}}
 \put(-3,0){\makebox(0,0)[r]{\normalsize$a$}}
 \put(43,0){\makebox(0,0)[l]{\normalsize$b$}}
\end{picture}
\end{minipage}
\end{center}
\vspace{0.5em}
at each vertex.
Although the axioms of a partially multiplicative biquandle are almost identical to the primitive conditions \eqref{eq:R4-1,primitive}--\eqref{eq:R6-4,primitive} under the correspondence
\begin{align*}
&a\bullet b=b\triangle^{-1}a=a(b\oline{*}^{-1}a), \\
&\widetilde{P}=\{(a,b\triangle a)\,|\,(b,a)\in P\}=\{(a,a^{-1}b\oline{*}a)\,|\,(a,b)\in\textstyle\bigsqcup_{\lambda\in\Lambda}G_\lambda^2\},
\end{align*}
the axiom (i) is an additional axiom to simplified the axioms.
Fortunately, we see that the axiom (i) is a necessary condition as follows.
By Theorem~\ref{thm:universality}, a partially multiplicative biquandle consists of a multiple conjugation biquandle and a biquandle.
Then the axiom (i) follows from Lemma~\ref{lem:x*e=x} (3), since we have $a^{-1}(a\bullet b)\oline{*}a=(a\bullet b)\triangle a=b$.

We prove Theorem~\ref{thm:universality} (1), (2) below, and (3) in the next section.

\begin{proof}
\begin{itemize}
\item[(1)]
We show that $\uline{*}x:X_1\to X_1$ is a well-defined bijection for any $x\in X$.
For any $b\in X_1$, there exists $a\in X$ such that $a\sim b$.
By \eqref{eq:R5-1,primitive}, we have $a\uline{*}x\sim b\uline{*}x$ and $a\uline{*}^{-1}x\sim b\uline{*}^{-1}x$, which imply $b\uline{*}x,b\uline{*}^{-1}x\in X_1$.
Therefore $\uline{*}x,\uline{*}^{-1}x:X_1\to X_1$ are well-defined bijections.
In the same way, we see that $\oline{*}x:X_1\to X_1$ is a well-defined bijection for any $x\in X$.
Since
\begin{align*}
&\uline{*}x,\oline{*}x:X_1\to X_1,
&&\uline{*}x,\oline{*}x:X_1\sqcup X_2\to X_1\sqcup X_2
\end{align*}
are bijections, $\uline{*}x,\oline{*}x:X_2\to X_2$ are well-defined bijections.
On $X\times X=(X_1\times X_1)\sqcup(X_1\times X_2)\sqcup(X_2\times X_1)\sqcup(X_2\times X_2)$, the bijection $S:X\times X\to X\times X$ defined by $S(x,y)=(y\oline{*}x,x\uline{*}y)$ is decomposed into the four bijections
\begin{align*}
&S:X_1\times X_1\to X_1\times X_1,
&&S:X_1\times X_2\to X_2\times X_1, \\
&S:X_2\times X_1\to X_1\times X_2,
&&S:X_2\times X_2\to X_2\times X_2.
\end{align*}
Therefore $X_1,X_2$ are subbiquandles of $X$.

\item[(2)]
For any $a\in X_1$, there exists $b\in X$ such that $b\sim a$ by the assumption.
From \eqref{eq:R6-3,primitive}, $b\sim a,b\sim a,x=b\triangle a\Rightarrow a\sim a$.
By \eqref{eq:R6-3,primitive}, $a\sim b$ with $a\sim a$ implies $b\sim a$.
Suppose $a\sim b$, $b\sim c$.
By \eqref{eq:R6-1,primitive}, we have $a\sim c$.
Thus $\sim$ is an equivalence relation on $X_1$.
\end{itemize}
\end{proof}

\section{Proof of Theorem~\ref{thm:universality} (3)}

We introduce the notion of a triangle MCB.
Although it is defined as the disjoint union of sets, it turns out that a triangle MCB consists of the disjoint union of groups.
Furthermore, we show that a triangle MCB is an MCB.
At the end of this section, we prove Theorem~\ref{thm:universality} (3).

\begin{definition}
A \textit{triangle MCB} $X=\bigsqcup_{\lambda\in\Lambda}G_\lambda$ is a biquandle $(X,\uline{*},\oline{*})$ with a map $\triangle:\bigsqcup_{\lambda\in\Lambda}G_\lambda^2\to X$ satisfying \eqref{eq:bijection,triangle}--\eqref{eq:R6,triangle}, where $G_\lambda$ is not necessarily a group.
\end{definition}

\begin{lemma} \label{lem:unique identity and inverse}
Let $X=\bigsqcup_{\lambda\in\Lambda}G_\lambda$ be a triangle MCB.
For $a,b\in G_\lambda$, we have
\begin{align}
&a\uline{*}b\triangle^{-1}b=b\oline{*}a\triangle^{-1}a
\in G_\lambda, \label{eq:2ab} \\
&a\triangle a\uline{*}^{-1}a=b\triangle b\uline{*}^{-1}b
\in G_\lambda. \label{eq:2e}
\end{align}
For $a\in G_\lambda$, we have
\begin{align}
&a\triangle a\uline{*}^{-1}a=\alpha\triangle\alpha
=a\triangle a\oline{*}^{-1}a, \label{eq:uo2e} \\
&a\triangle a\uline{*}^{-1}a\triangle a\uline{*}^{-1}a
=a\triangle a\oline{*}^{-1}a\triangle a\oline{*}^{-1}a
\in G_\lambda, \label{eq:uo2inverse}
\end{align}
where $\alpha\in X$ is the unique element satisfying
$\alpha\uline{*}\alpha=\alpha\oline{*}\alpha=a$.
\end{lemma}

\begin{proof}
The equality \eqref{eq:2ab} follows from
\begin{align*}
&a\uline{*}b\triangle^{-1}b
=(a\uline{*}b\triangle^{-1}b)\triangle(a\triangle b\triangle^{-1}b)\triangle^{-1}a \\
&\overset{\eqref{eq:R6,triangle}}{=}
(a\uline{*}b)\triangle(a\triangle b)\triangle^{-1}a
\overset{\eqref{eq:R4,u,triangle}}{=}
b\oline{*}a\triangle^{-1}a.
\end{align*}
The equality \eqref{eq:2e} follows from
\begin{align*}
a\triangle a\uline{*}^{-1}a
&=(a\triangle a\uline{*}^{-1}a)\uline{*}b\uline{*}^{-1}b \\
&\overset{\eqref{eq:R5-2,u,triangle}}{=}
(a\triangle a)\uline{*}(b\triangle a)\uline{*}^{-1}b \\
&=(a\triangle a)\uline{*}(b\triangle a)\triangle^{-1}(b\triangle a)\triangle(b\triangle a)\uline{*}^{-1}b \\
&\overset{\eqref{eq:2ab}}{=}
(b\triangle a)\oline{*}(a\triangle a)\triangle^{-1}(a\triangle a)\triangle(b\triangle a)\uline{*}^{-1}b \\
&\overset{\eqref{eq:R5-2,o,triangle}}{=}
(b\triangle a\oline{*}^{-1}a)\oline{*}a\triangle^{-1}(a\triangle a)\triangle(b\triangle a)\uline{*}^{-1}b \\
&=(b\triangle a)\triangle^{-1}(a\triangle a)\triangle(b\triangle a)\uline{*}^{-1}b \\
&\overset{\eqref{eq:R6,triangle}}{=}
(b\triangle a)\triangle(a\triangle a)\triangle^{-1}(a\triangle a)\triangle(b\triangle a)\uline{*}^{-1}b \\
&=(b\triangle a)\triangle(b\triangle a)\uline{*}^{-1}b \\
&\overset{\eqref{eq:R6,triangle}}{=}
b\triangle b\uline{*}^{-1}b.
\end{align*}
Then \eqref{eq:uo2e} follows from
\begin{align*}
&a\triangle a\uline{*}^{-1}a
=(\alpha\uline{*}\alpha)\triangle(\alpha\uline{*}\alpha)\uline{*}^{-1}a
\overset{\eqref{eq:R5-1,u,triangle}}{=}
(\alpha\triangle\alpha)\uline{*}(\alpha\oline{*}\alpha)\uline{*}^{-1}a
=\alpha\triangle\alpha, \\
&a\triangle a\oline{*}^{-1}a
=(\alpha\oline{*}\alpha)\triangle(\alpha\oline{*}\alpha)\oline{*}^{-1}a
\overset{\eqref{eq:R5-1,o,triangle}}{=}
(\alpha\triangle\alpha)\oline{*}(\alpha\uline{*}\alpha)\oline{*}^{-1}a
=\alpha\triangle\alpha.
\end{align*}
The equality \eqref{eq:uo2inverse} follows from \eqref{eq:uo2e} and
\begin{align*}
&a\triangle a\uline{*}^{-1}a\triangle a\uline{*}^{-1}a
\overset{\eqref{eq:uo2e}}{=}
\alpha\triangle\alpha\triangle a\uline{*}^{-1}a
=(\alpha\triangle\alpha\uline{*}^{-1}\alpha\uline{*}\alpha)\triangle(\alpha\uline{*}\alpha)\uline{*}^{-1}a \\
&\overset{\eqref{eq:R5-1,u,triangle}}{=}
(\alpha\triangle\alpha\uline{*}^{-1}\alpha\triangle\alpha)\uline{*}(\alpha\oline{*}\alpha)\uline{*}^{-1}a
=\alpha\triangle\alpha\uline{*}^{-1}\alpha\triangle\alpha, \\
&a\triangle a\oline{*}^{-1}a\triangle a\oline{*}^{-1}a
\overset{\eqref{eq:uo2e}}{=}
\alpha\triangle\alpha\triangle a\oline{*}^{-1}a
=(\alpha\triangle\alpha\oline{*}^{-1}\alpha\oline{*}\alpha)\triangle(\alpha\oline{*}\alpha)\oline{*}^{-1}a \\
&\overset{\eqref{eq:R5-1,o,triangle}}{=}
(\alpha\triangle\alpha\oline{*}^{-1}\alpha\triangle\alpha)\oline{*}(\alpha\uline{*}\alpha)\oline{*}^{-1}a
=\alpha\triangle\alpha\oline{*}^{-1}\alpha\triangle\alpha.
\end{align*}
We have $b\oline{*}a\triangle^{-1}a,a\triangle a\uline{*}^{-1}a,a\triangle a\uline{*}^{-1}a\triangle a\uline{*}^{-1}a\in G_\lambda$, since $\uline{*}a,\oline{*}a,\triangle a$ are bijections from $G_\lambda$ to $G_{a\uline{*}a}=G_{a\oline{*}a}=G_{a\triangle a}$.
\end{proof}

\begin{proposition} \label{prop:triangle2MCB}
Let $X=\bigsqcup_{\lambda\in\Lambda}G_\lambda$ be a triangle MCB.
\begin{itemize}
\item[(1)]
For any $\lambda\in\Lambda$, $G_\lambda$ is a group with
\begin{align*}
&ab:=a\uline{*}b\triangle^{-1}b=b\oline{*}a\triangle^{-1}a\in G_\lambda, \\
&e_\lambda:=a\triangle a\uline{*}^{-1}a=a\triangle a\oline{*}^{-1}a\in G_\lambda, \\
&a^{-1}:=a\triangle a\uline{*}^{-1}a\triangle a\uline{*}^{-1}a
=a\triangle a\oline{*}^{-1}a\triangle a\oline{*}^{-1}a\in G_\lambda
\end{align*}
for $a,b\in G_\lambda$.

\item[(2)]
The triangle MCB $X=\bigsqcup_{\lambda\in\Lambda}G_\lambda$ is a multiple conjugation biquandle.
\end{itemize}
\end{proposition}

\begin{proof}
\begin{itemize}
\item[(1)]
By Lemma~\ref{lem:unique identity and inverse}, the multiplication, identity, and inverse are well-defined.
The associativity $(ab)c=a(bc)$ follows from
\begin{align*}
&(ab)c\triangle(ab)
=c\oline{*}(ab)
=c\oline{*}(b\oline{*}a\triangle^{-1}a)
\overset{\eqref{eq:R5-2,o,triangle}}{=}
(c\oline{*}a)\oline{*}(b\oline{*}a) \\
&\overset{\rm(B3)}{=}
(c\oline{*}b)\oline{*}(a\uline{*}b)
=(bc\triangle b)\oline{*}(a\uline{*}b)
\overset{\eqref{eq:R5-1,o,triangle}}{=}
(bc\oline{*}a)\triangle(b\oline{*}a) \\
&\overset{\eqref{eq:R6,triangle}}{=}
(bc\oline{*}a\triangle^{-1}a)\triangle(b\oline{*}a\triangle^{-1}a)
=a(bc)\triangle(ab).
\end{align*}
We have
\begin{align*}
&e_\lambda a=(a\triangle a\uline{*}^{-1}a)\uline{*}a\triangle^{-1}a=a, \\
&ae_\lambda=(a\triangle a\oline{*}^{-1}a)\oline{*}a\triangle^{-1}a=a.
\end{align*}
We have
\begin{align*}
&a^{-1}a
=(a\triangle a\uline{*}^{-1}a\triangle a\uline{*}^{-1}a)\uline{*}a\triangle^{-1}a
=a\triangle a\uline{*}^{-1}a
=e_\lambda, \\
&aa^{-1}
=(a\triangle a\oline{*}^{-1}a\triangle a\oline{*}^{-1}a)\oline{*}a\triangle^{-1}a
=a\triangle a\oline{*}^{-1}a
=e_\lambda.
\end{align*}

\item[(2)]
The maps $\uline{*}x:G_a\to G_{a\uline{*}x}$ and $\oline{*}x:G_a\to G_{a\oline{*}x}$ are group homomorphism, since $b^{-1}a\uline{*}x=(b\uline{*}x)^{-1}(a\uline{*}x)$ and $b^{-1}a\oline{*}x=(b\oline{*}x)^{-1}(a\oline{*}x)$ follow from
\begin{align*}
&(b\uline{*}x)^{-1}(a\uline{*}x)\oline{*}(b\uline{*}x)
=(a\uline{*}x)\triangle(b\uline{*}x) \\
&\overset{\eqref{eq:R5-1,u,triangle}}{=}
(a\triangle b)\uline{*}(x\oline{*}b)
=(b^{-1}a\oline{*}b)\uline{*}(x\oline{*}b)
\overset{\rm(B3)}{=}
(b^{-1}a\uline{*}x)\oline{*}(b\uline{*}x), \\
&(b\oline{*}x)^{-1}(a\oline{*}x)\oline{*}(b\oline{*}x)
=(a\oline{*}x)\triangle(b\oline{*}x) \\
&\overset{\eqref{eq:R5-1,o,triangle}}{=}
(a\triangle b)\oline{*}(x\uline{*}b)
=(b^{-1}a\oline{*}b)\oline{*}(x\uline{*}b)
\overset{\rm(B3)}{=}
(b^{-1}a\oline{*}x)\oline{*}(b\oline{*}x).
\end{align*}
For $a,b\in G_\lambda$ and $x\in X$, we have
\begin{align*}
&x\uline{*}ab
\overset{\eqref{eq:R5-2,u,triangle}}{=}
(x\uline{*}a)\uline{*}(ab\triangle a)
=(x\uline{*}a)\uline{*}(b\oline{*}a), \\
&x\oline{*}ab
\overset{\eqref{eq:R5-2,o,triangle}}{=}
(x\oline{*}a)\oline{*}(ab\triangle a)
=(x\oline{*}a)\oline{*}(b\oline{*}a),
\end{align*}
and
\begin{align*}
a^{-1}b\oline{*}a
&\overset{\eqref{eq:R5-2,o,triangle}}{=}
(a^{-1}b\oline{*}ba^{-1})\oline{*}(a\triangle ba^{-1}) \\
&=(a^{-1}b\oline{*}ba^{-1})\oline{*}(ab^{-1}a\oline{*}ba^{-1}) \\
&\overset{\rm(B3)}{=}
(a^{-1}b\oline{*}ab^{-1}a)\oline{*}(ba^{-1}\uline{*}ab^{-1}a) \\
&=(a\triangle ab^{-1}a)\oline{*}(ba^{-1}\uline{*}ab^{-1}a) \\
&\overset{\eqref{eq:R5-1,o,triangle}}{=}
(a\oline{*}ba^{-1})\triangle(ab^{-1}a\oline{*}ba^{-1}) \\
&=(a\oline{*}ba^{-1})\triangle(a\triangle ba^{-1}) \\
&\overset{\eqref{eq:R4,o,triangle}}{=}
ba^{-1}\uline{*}a.
\end{align*}
\end{itemize}
\end{proof}

\begin{proof}[Proof of Theorem~\ref{thm:universality} (3)]
By Proposition~\ref{prop:triangle2MCB}, it is sufficient to show that $X_1$ is a triangle MCB.
We show \eqref{eq:bijection,triangle}, \eqref{eq:bijection,uo}, and $G_{a\triangle b}=G_{a\uline{*}b}=G_{a\oline{*}b}$ for $a,b\in G_\lambda$.
The other equalities \eqref{eq:R4,u,triangle}--\eqref{eq:R6,triangle} follow directly from the primitive conditions \eqref{eq:R4-1,primitive}--\eqref{eq:R6-4,primitive}.
For $a,b\in G_\lambda$, $a\sim b$ implies $a\uline{*}b\sim a\triangle b$ and $a\oline{*}b\sim a\triangle b$ by \eqref{eq:R4-1,primitive} and \eqref{eq:R4-2,primitive}, respectively.
Then $G_{a\triangle b}=G_{a\uline{*}b}=G_{a\oline{*}b}$.

We verify \eqref{eq:bijection,triangle}.
The map $\triangle a:G_a\to G_{a\triangle a}$ is well-defined, since $x\triangle a\sim a\triangle a$ follows from $x\sim a$ and $a\sim a$ by \eqref{eq:R6-1,primitive}.
Let $y\in G_{a\triangle a}$.
Then $a\sim a,y\sim a\triangle a$.
By \eqref{eq:R6-4,primitive}, 
\[ \exists!x\in X\text{ s.t.~}x\sim a,y=x\triangle a,y\triangle(a\triangle a)=x\triangle a. \]
Since $(x\triangle a)\triangle(a\triangle a)=x\triangle a$ follows from \eqref{eq:R6-1,primitive}, we can remove the condition $y\triangle(a\triangle a)=x\triangle a$, that is,
\[ \exists!x\in X\text{ s.t.~}x\sim a,y=x\triangle a. \]
Then $\triangle a$ is bijective.

We verify \eqref{eq:bijection,uo}.
By \eqref{eq:R5-1,primitive}, $\uline{*}x:G_a\to G_{a\uline{*}x}$ and $\uline{*}^{-1}x:G_{a\uline{*}x}\to G_a$ are well-defined.
By \eqref{eq:R5-2,primitive}, $\oline{*}x:G_a\to G_{a\oline{*}x}$ and $\oline{*}^{-1}x:G_{a\oline{*}x}\to G_a$ are well-defined.
Therefore $\uline{*}x:G_a\to G_{a\uline{*}x}$ and $\oline{*}x:G_a\to G_{a\oline{*}x}$ are well-defined bijections.
\end{proof}

\section{Parallel biquandle operations}
\label{sect:parallel}

In this section, we show that the $n$-parallel biquandle operations are well-defined and that $(X,(\uline{*}^{[n]})_{n\in\mathbb{Z}},(\oline{*}^{[n]})_{n\in\mathbb{Z}})$ is a $\mathbb{Z}$-family of biquandles.

\begin{proposition}
The binary operations $\uline{*}^{[n]},\oline{*}^{[n]}:X\times X\to X$ are well-defined for any $n\in\mathbb{Z}$.
\end{proposition}

\begin{proof}
Let $\varphi:X\times X\to X\times X$ be the bijection defined by $\varphi(x,y)=(x\uline{*}y,y\uline{*}y)$, where the bijectivity follows from Lemma~\ref{lem:x*x=y}.
For $n\in\mathbb{Z}$, we define $f_n,g_n:X\times X\to X$ by $\varphi^n(x,y)=(f_n(x,y),g_n(x,y))$.
Then
\begin{align*}
&f_{n+1}(x,y)=f_n(x,y)\uline{*}g_n(x,y),
&&g_{n+1}(x,y)=g_n(x,y)\uline{*}g_n(x,y).
\end{align*}
We show that $\uline{*}^{[n]},f_n:X\times X\to X$ coincide.
Since $a\uline{*}^{[n]}b$ can be calculated by using \eqref{eq:*u[n]def}, it is sufficient to show the equalities
\begin{align*}
&f_0(a,b)=a,
&&f_1(a,b)=a\uline{*}b,
&&f_{i+j}(a,b)=f_j(f_i(a,b),f_i(b,b)),
\end{align*}
which correspond to \eqref{eq:*u[n]def}.

We show the equality $g_n(x,y)=f_n(y,y)$ by induction on $n$.
When $n=0$, the both sides coincide with $y$.
We assume that the equality holds when $n=k$ for some $k\in\mathbb{Z}_{\geq0}$.
Then we have
\begin{align*}
&g_{k+1}(x,y)=g_k(x,y)\uline{*}g_k(x,y)=f_k(y,y)\uline{*}f_k(y,y) \\
&=f_k(y,y)\uline{*}g_k(y,y)=f_{k+1}(y,y).
\end{align*}
We assume that the equality holds when $n=-k$ for some $k\in\mathbb{Z}_{\geq0}$.
By Lemma~\ref{lem:x*x=y}, we have $g_{-k-1}(y,y)=g_{-k-1}(x,y)$ from
\begin{align*}
&g_{-k-1}(y,y)\uline{*}g_{-k-1}(y,y)=g_{-k}(y,y)=f_{-k}(y,y) \\
&=g_{-k}(x,y)=g_{-k-1}(x,y)\uline{*}g_{-k-1}(x,y).
\end{align*}
Then the equality $g_{-k-1}(x,y)=f_{-k-1}(y,y)$ follows from
\begin{align*}
&g_{-k-1}(x,y)\uline{*}g_{-k-1}(x,y)=g_{-k}(x,y)=f_{-k}(y,y) \\
&=f_{-k-1}(y,y)\uline{*}g_{-k-1}(y,y)=f_{-k-1}(y,y)\uline{*}g_{-k-1}(x,y).
\end{align*}

Then we have
\[ f_j(f_i(a,b),f_i(b,b))=f_j(f_i(a,b),g_i(a,b))=f_{i+j}(a,b), \]
where the last equality follows from
\begin{align*}
&(f_{i+j}(x,y),g_{i+j}(x,y))
=\varphi^{i+j}(x,y)
=\varphi^j(\varphi^i(x,y)) \\
&=\varphi^j(f_i(x,y),g_i(x,y))
=(f_j(f_i(x,y),g_i(x,y)),g_j(f_i(x,y),g_i(x,y))).
\end{align*}
Therefore $\uline{*}^{[n]}$ coincides with $f_n$, which is well-defined.
In a similar manner, we see that $\oline{*}^{[n]}$ is well-defined.
\end{proof}

\begin{lemma} \label{lem:[n]inverse}
For $n\in\mathbb{Z}$, we have the following.
\begin{itemize}
\item
If $a\uline{*}^{[n]}b=c$, then $a=c\uline{*}^{[-n]}(b\uline{*}^{[n]}b)$.
\item
If $a\oline{*}^{[n]}b=c$, then $a=c\oline{*}^{[-n]}(b\oline{*}^{[n]}b)$.
\end{itemize}
In particular, for $n\in\mathbb{Z}$, we have the following.
\begin{itemize}
\item
If $a\uline{*}^{[n]}a=c$, then $a=c\uline{*}^{[-n]}c$.
\item
If $a\oline{*}^{[n]}a=c$, then $a=c\oline{*}^{[-n]}c$.
\end{itemize}
\end{lemma}

\begin{proof}
We have
$a=a\uline{*}^{[0]}b
=(a\uline{*}^{[n]}b)\uline{*}^{[-n]}(b\uline{*}^{[n]}b)
=c\uline{*}^{[-n]}(b\uline{*}^{[n]}b)$.
If $a=b$, then $a=c\uline{*}^{[-n]}(b\uline{*}^{[n]}b)=c\uline{*}^{[-n]}c$.
In the same way, we see the remaining part.
\end{proof}

\begin{proposition} \label{prop:Z-family}
Let $(X,\uline{*},\oline{*})$ be a biquandle.
Then $(X,(\uline{*}^{[n]})_{n\in\mathbb{Z}},(\oline{*}^{[n]})_{n\in\mathbb{Z}})$ is a $\mathbb{Z}$-family of biquandles.
\end{proposition}

\begin{proof}
We show $a\uline{*}^{[n]}a=a\oline{*}^{[n]}a$ for $n\in\mathbb{Z}$ and $a\in X$.
Let $f,g:X\to X$ be the bijections defined by $f(x)=x\uline{*}x$, $g(x)=x\oline{*}x$.
Then we have $f^n(x)=x\uline{*}^{[n]}x$ and $g^n(x)=x\oline{*}^{[n]}x$.
Since $f=g$ follows from (B1), we have
\[ a\uline{*}^{[n]}a=f^n(a)=g^n(a)=a\oline{*}^{[n]}a. \]
Then, by the definition of $\uline{*}^{[n]}$ and $\oline{*}^{[n]}$, it is sufficient to show that
\begin{align}
(a\uline{*}^{[m]}b)\uline{*}^{[n]}(c\oline{*}^{[m]}b)
&=(a\uline{*}^{[n]}c)\uline{*}^{[m]}(b\uline{*}^{[n]}c),
\label{eq:R3-1,mn} \\
(a\oline{*}^{[m]}b)\uline{*}^{[n]}(c\oline{*}^{[m]}b)
&=(a\uline{*}^{[n]}c)\oline{*}^{[m]}(b\uline{*}^{[n]}c),
\label{eq:R3-2,mn} \\
(a\oline{*}^{[m]}b)\oline{*}^{[n]}(c\oline{*}^{[m]}b)
&=(a\oline{*}^{[n]}c)\oline{*}^{[m]}(b\uline{*}^{[n]}c)
\label{eq:R3-3,mn}
\end{align}
for $m,n\in\mathbb{Z}$ and $a,b,c\in X$.
These equalities were verified for $m,n\geq0$ in \cite{IshiiNelson16}.
We note that
\begin{align*}
&x=x\uline{*}^{[0]}y=(x\uline{*}^{[n]}y)\uline{*}^{[-n]}(y\uline{*}^{[n]}y)=(x\uline{*}^{[n]}y)\uline{*}^{[-n]}(y\oline{*}^{[n]}y), \\
&x=x\oline{*}^{[0]}y=(x\oline{*}^{[n]}y)\oline{*}^{[-n]}(y\oline{*}^{[n]}y)=(x\oline{*}^{[n]}y)\oline{*}^{[-n]}(y\uline{*}^{[n]}y).
\end{align*}

We show the equality \eqref{eq:R3-2,mn}.
Let $m\geq0$, $n=-k\leq0$.
By Lemma~\ref{lem:[n]inverse}, the equality
\begin{align}
(c\uline{*}^{[-k]}c)\oline{*}^{[m]}(b\uline{*}^{[-k]}c)
=(c\oline{*}^{[m]}b)\uline{*}^{[-k]}(c\oline{*}^{[m]}b)
\label{eq:R3,aba-1}
\end{align}
follows from
\begin{align*}
c\oline{*}^{[m]}b
&=((c\uline{*}^{[-k]}c)\uline{*}^{[k]}(c\uline{*}^{[-k]}c))\oline{*}^{[m]}((b\uline{*}^{[-k]}c)\uline{*}^{[k]}(c\uline{*}^{[-k]}c)) \\
&=((c\uline{*}^{[-k]}c)\oline{*}^{[m]}(b\uline{*}^{[-k]}c))\uline{*}^{[k]}((c\uline{*}^{[-k]}c)\oline{*}^{[m]}(b\uline{*}^{[-k]}c)).
\end{align*}
Then we have
\begin{align*}
a\oline{*}^{[m]}b
&=((a\uline{*}^{[-k]}c)\uline{*}^{[k]}(c\uline{*}^{[-k]}c))\oline{*}^{[m]}((b\uline{*}^{[-k]}c)\uline{*}^{[k]}(c\uline{*}^{[-k]}c)) \\
&=((a\uline{*}^{[-k]}c)\oline{*}^{[m]}(b\uline{*}^{[-k]}c))\uline{*}^{[k]}((c\uline{*}^{[-k]}c)\oline{*}^{[m]}(b\uline{*}^{[-k]}c)) \\
&\overset{\eqref{eq:R3,aba-1}}{=}
((a\uline{*}^{[-k]}c)\oline{*}^{[m]}(b\uline{*}^{[-k]}c))\uline{*}^{[k]}((c\oline{*}^{[m]}b)\uline{*}^{[-k]}(c\oline{*}^{[m]}b))
\end{align*}
By Lemma~\ref{lem:[n]inverse}, we have
\begin{align*}
&(a\uline{*}^{[-k]}c)\oline{*}^{[m]}(b\uline{*}^{[-k]}c) \\
&=(a\oline{*}^{[m]}b)\uline{*}^{[-k]}(((c\oline{*}^{[m]}b)\uline{*}^{[-k]}(c\oline{*}^{[m]}b))\uline{*}^{[k]}((c\oline{*}^{[m]}b)\uline{*}^{[-k]}(c\oline{*}^{[m]}b))) \\
&=(a\oline{*}^{[m]}b)\uline{*}^{[-k]}(c\oline{*}^{[m]}b).
\end{align*}

Let $m=-k\leq0$, $n\geq0$.
By Lemma~\ref{lem:[n]inverse}, the equality
\begin{align}
(b\oline{*}^{[-k]}b)\uline{*}^{[n]}(c\oline{*}^{[-k]}b)
=(b\uline{*}^{[n]}c)\oline{*}^{[-k]}(b\uline{*}^{[n]}c)
\label{eq:R3,aba-2}
\end{align}
follows from
\begin{align*}
b\uline{*}^{[n]}c
&=((b\oline{*}^{[-k]}b)\oline{*}^{[k]}(b\oline{*}^{[-k]}b))\uline{*}^{[n]}((c\oline{*}^{[-k]}b)\oline{*}^{[k]}(b\oline{*}^{[-k]}b)) \\
&=((b\oline{*}^{[-k]}b)\uline{*}^{[n]}(c\oline{*}^{[-k]}b))\oline{*}^{[k]}((b\oline{*}^{[-k]}b)\uline{*}^{[n]}(c\oline{*}^{[-k]}b)).
\end{align*}
Then we have
\begin{align*}
a\uline{*}^{[n]}c
&=((a\oline{*}^{[-k]}b)\oline{*}^{[k]}(b\oline{*}^{[-k]}b))\uline{*}^{[n]}((c\oline{*}^{[-k]}b)\oline{*}^{[k]}(b\oline{*}^{[-k]}b)) \\
&=((a\oline{*}^{[-k]}b)\uline{*}^{[n]}(c\oline{*}^{[-k]}b))\oline{*}^{[k]}((b\oline{*}^{[-k]}b)\uline{*}^{[n]}(c\oline{*}^{[-k]}b)) \\
&\overset{\eqref{eq:R3,aba-2}}{=}
((a\oline{*}^{[-k]}b)\uline{*}^{[n]}(c\oline{*}^{[-k]}b))\oline{*}^{[k]}((b\uline{*}^{[n]}c)\oline{*}^{[-k]}(b\uline{*}^{[n]}c)).
\end{align*}
By Lemma~\ref{lem:[n]inverse}, we have
\begin{align*}
&(a\oline{*}^{[-k]}b)\uline{*}^{[n]}(c\oline{*}^{[-k]}b) \\
&=(a\uline{*}^{[n]}c)\oline{*}^{[-k]}(((b\uline{*}^{[n]}c)\oline{*}^{[-k]}(b\uline{*}^{[n]}c))\oline{*}^{[k]}((b\uline{*}^{[n]}c)\oline{*}^{[-k]}(b\uline{*}^{[n]}c))) \\
&=(a\uline{*}^{[n]}c)\oline{*}^{[-k]}(b\uline{*}^{[n]}c).
\end{align*}

Let $m=-k\leq0$, $n=-l\leq0$.
By Lemma~\ref{lem:[n]inverse}, the equality
\begin{align}
(b\oline{*}^{[-k]}b)\uline{*}^{[-l]}(c\oline{*}^{[-k]}b)
=(b\uline{*}^{[-l]}c)\oline{*}^{[-k]}(b\uline{*}^{[-l]}c)
\label{eq:R3,aba-3}
\end{align}
follows from
\begin{align*}
b\uline{*}^{[-l]}c
&=((b\oline{*}^{[-k]}b)\oline{*}^{[k]}(b\oline{*}^{[-k]}b))\uline{*}^{[-l]}((c\oline{*}^{[-k]}b)\oline{*}^{[k]}(b\oline{*}^{[-k]}b)) \\
&=((b\oline{*}^{[-k]}b)\uline{*}^{[-l]}(c\oline{*}^{[-k]}b))\oline{*}^{[k]}((b\oline{*}^{[-k]}b)\uline{*}^{[-l]}(c\oline{*}^{[-k]}b)).
\end{align*}
Then we have
\begin{align*}
a\uline{*}^{[-l]}c
&=((a\oline{*}^{[-k]}b)\oline{*}^{[k]}(b\oline{*}^{[-k]}b))\uline{*}^{[-l]}((c\oline{*}^{[-k]}b)\oline{*}^{[k]}(b\oline{*}^{[-k]}b)) \\
&=((a\oline{*}^{[-k]}b)\uline{*}^{[-l]}(c\oline{*}^{[-k]}b))\oline{*}^{[k]}((b\oline{*}^{[-k]}b)\uline{*}^{[-l]}(c\oline{*}^{[-k]}b)) \\
&\overset{\eqref{eq:R3,aba-3}}{=}
((a\oline{*}^{[-k]}b)\uline{*}^{[-l]}(c\oline{*}^{[-k]}b))\oline{*}^{[k]}((b\uline{*}^{[-l]}c)\oline{*}^{[-k]}(b\uline{*}^{[-l]}c)).
\end{align*}
By Lemma~\ref{lem:[n]inverse}, we have
\begin{align*}
&(a\oline{*}^{[-k]}b)\uline{*}^{[-l]}(c\oline{*}^{[-k]}b) \\
&=(a\uline{*}^{[-l]}c)\oline{*}^{[-k]}(((b\uline{*}^{[-l]}c)\oline{*}^{[-k]}(b\uline{*}^{[-l]}c))\oline{*}^{[k]}((b\uline{*}^{[-l]}c)\oline{*}^{[-k]}(b\uline{*}^{[-l]}c))) \\
&=(a\uline{*}^{[-l]}c)\oline{*}^{[-k]}(b\uline{*}^{[-l]}c).
\end{align*}
This completes the proof of \eqref{eq:R3-2,mn}.
In a similar manner, we can verify \eqref{eq:R3-1,mn} and \eqref{eq:R3-3,mn} by using the equalities \eqref{eq:R3,aba-1}--\eqref{eq:R3,aba-3}.

This completes the proof.
\end{proof}

%

\end{document}